\documentclass[11pt]{amsart}
\usepackage{geometry}                
\usepackage{graphicx}
\usepackage{amsmath,amsthm,amssymb}
\usepackage{epstopdf}
\usepackage {tikz}
\usepackage{bm}
\usepackage{hyperref}
\usepackage{ytableau} 
\usepackage{young}  
\usepackage[enableskew]{youngtab} 
\usepackage{pgfplots}
\usepackage{caption}
\usepgfplotslibrary{ternary}
\usepackage{subcaption}
\usetikzlibrary{decorations.markings}
\usetikzlibrary{arrows.meta}
\usetikzlibrary{shapes.geometric}

\usepackage{tikz-cd}

\usepackage{enumitem}
  
\pgfplotsset{compat=1.13}

\newcommand\scalemath[2]{\scalebox{#1}{\mbox{\ensuremath{\displaystyle #2}}}}

\newcommand\xleftrightarrow[2][]{\ext@arrow 0099{\longleftrightarrowfill@}{#1}{#2}}
\def\longleftrightarrowfill@{\arrowfill@\leftarrow\relbar\rightarrow}

\let\emph\relax 
\DeclareTextFontCommand{\emph}{\bfseries\em}

\DeclareMathOperator{\Inv}{Inv}

\newtheorem{Thm}{Theorem}
\newtheorem{Prop}[Thm]{Proposition}
\newtheorem{Lem}[Thm]{Lemma}
\newtheorem{Cor}[Thm]{Corollary}

\theoremstyle{remark}
\newtheorem{Rem}[Thm]{Remark}
\newtheorem{Ex}[Thm]{Example}

\theoremstyle{definition}
\newtheorem{Def}[Thm]{Definition}

\title{Promotion and Cyclic Sieving on Rectangular $\delta$-semistandard Tableaux}
\author{Tair Akhmejanov and Bal\'azs Elek}
\begin{document}
\maketitle


\begin{abstract}
Let $\delta=(\delta_1,\ldots,\delta_n)$ be a string of letters $h$ and $v$. We define a Young tableau to be $\delta$-semistandard if the entries are weakly increasing along rows and columns, and the entries $i$ form a horizontal strip if $\delta_i=h$ and a vertical strip if $\delta_i=v$. We define $\delta$-promotion on such tableaux via a modified jeu-de-taquin. The first main result is that $\delta$-promotion has period $n$ on rectangular $\delta$-semistandard tableaux, generalizing the results of Haiman and Rhoades for standard and semistandard tableaux. The second main result states that the set of rectangular $\delta$-semistandard tableaux for fixed $\delta$ and content $\gamma$ exhibits the cyclic sieving phenomenon with the generalized Kostka polynomial. To do so we follow Fontaine--Kamnitzer and associate to $(\delta,\gamma)$ an $SL_m$-invariant space $\Inv(V_{\lambda^1}\otimes\cdots\otimes V_{\lambda^n})$ where each $V_{\lambda^i}$ is an alternating or symmetric representation. We show that the Satake basis of the corresponding invariant space is indexed by the set of tableaux corresponding to $(\delta,\gamma)$ and is permuted by rotation of tensor factors. We then diagonalize the rotation action using the fusion product. This cyclic sieving generalizes the result of Rhoades, and of Fontaine-Kamnitzer (in type A), and is closely related to that of Westbury.
\end{abstract}

\tableofcontents

\section{Introduction}

In this paper we define $\delta$-semistandard tableaux where $\delta=(\delta_1,\ldots,\delta_n)$ is a binary string of letters $\bf h$ and $\bf v$. The letter $\delta_i$ specifies whether the $i$ entries of the tableau form a horizontal ($\bf h$) or vertical ($\bf v$) strip. The $\delta$-semistandard tableaux specialize to the usual semistandard tableaux when $\delta=(\bf h,\ldots,h)$. We generalize the Sch\"utzenberger promotion operator on semistandard tableaux to promotion on $\delta$-semistandard tableaux via a generalized jeu-de-taquin. The generalized $\delta$-promotion operator maps a $\delta$-semistandard tableau $T$ with content $\gamma=(\gamma_1,\ldots,\gamma_n)$ to a $R(\delta)$-semistandard tableau $\partial_\delta(T)$ with content $R(\gamma)$ where $R$ is rotation of strings defined by $R(\delta)=(\delta_2,\ldots,\delta_n,\delta_1)$. Our first main result is that $\delta$-promotion has period $n$ on the set of rectangular $\delta$-semistandard tableaux. This generalizes the well-known result for rectangular semistandard tableaux \cite{Haiman, Rho}. 

\begin{Thm}\label{thm:periodicity}
Fix an orientation string $\delta=(\delta_1,\ldots,\delta_n)$. For any rectangular $\delta$-semistandard tableau $T$ we have $\partial_{R^{n-1}(\delta)}\cdots\partial_{R(\delta)}\partial_\delta(T)=T$.
\end{Thm}

To prove this we define a generalized Bender--Knuth involution when $\delta_i\not=\delta_{i+1}$. The proof use Knutson--Tao hives, which can also be interpreted as Fock--Goncharov coordinates \cite{FG}. 

The second main result is an instance of the cyclic sieving phenomenon for the set of rectangular $\delta$-semistandard tableaux with fixed orientation $\delta$ and content $\gamma$, generalizing the result of Rhoades \cite{Rho} and of Fontaine--Kamnitzer \cite{FK} (when the results of \cite{FK} are restricted to type A). Note that Fontaine--Kamnitzer worked in the setting of any simply-laced group $G$, but we are only concerned with $G=SL_m$. Our result is closely related to that of Westbury \cite{Wes}, and an alternate proof using the techniques of \cite{Wes} is discussed in Remark \ref{rem:Westbury} below. Westbury \cite{Wes} proved a general cyclic sieving on the set of invariant crystal elements in a tensor power $B^{\otimes n}$ of a crystal $B$ of a semi-simple Lie algebra (note that $B$ is not assumed to be connected). See also \cite{WesLoc} for an interpretation in terms of crystal local rules. 

To briefly recall cyclic sieving, let $X$ be a set with an action of the cyclic group $C_l=\langle c\rangle$, $\zeta$ a primitive $l$th root of unity, and $f(x)$ a polynomial. Then the triple $\left(X, c,f(x)\right)$ exhibits the cyclic sieving phenomenon if $f(\zeta^d)=\lvert X^{c^d}\rvert$ for all integers $d\geq 0$ where $X^{c^d}$ denotes the fixed-point set of $c^d$. See \cite{RSW, Sag} for details.

Fix orientation and content strings $\delta=(\delta_1,\ldots,\delta_n)$ and $\gamma=(\gamma_1,\ldots,\gamma_n)$, and let $r$ be the smallest positive integer such that $R^r(\delta)=\delta$ and $R^r(\gamma)=\gamma$. Set $l=n/r$ and $\lvert \gamma\rvert=\sum \gamma_i$. Let $RT_{m}(\delta,\gamma)$ denote the set of $\delta$-semistandard, $(m\times \frac{\lvert\gamma\rvert}{m})$-rectangular tableaux with content $\gamma$. Let the cyclic group $C_l$ act on $RT_m(\delta,\gamma)$ with generator $c=\partial_{R^{r-1}(\delta)}\cdots\partial_\delta$. To $(\delta,\gamma)$ associate the sequence of row and column partitions $\vec\lambda=(\lambda^1,\ldots,\lambda^n)$ where $\lambda^i$ is a row (resp. column) of size $\gamma_i$ if $\delta_i=\bf h$ (resp. $\delta_i=\bf v$). Let $\mu$ denote the $(m\times \frac{\lvert \gamma\rvert}{m})$-rectangular partition and $K_{\mu,\vec{\lambda}}(q)$ the generalized (cocharge) Kostka polynomial of \cite{KS,SS,SW}.

We show that the techniques of Fontaine, Kamnitzer \cite{FK} generalize easily to establish the following theorem. 
\begin{Thm}\label{thm:csp}
The following triple exhibits the cyclic sieving phenomenon.
\begin{align*}
\left(RT_{m}(\delta,\gamma), \partial_{R^{r-1}(\delta)}\cdots\partial_\delta, q^{\langle\lvert\vec\lambda\rvert,\rho\rangle}K_{\mu,\vec{\lambda}}(q)\right)
\end{align*}
\end{Thm}
Here $\rho$ is the half-sum of positive roots of $\mathfrak{sl}_m$, $\lvert \vec\lambda\rvert=\sum \lambda^i$, and we interpret partitions as dominant $\mathfrak{sl}_m$ weights. Note that the generalized Kostka polynomial is equivalent to the graded multiplicity of the trivial representation in the fusion product of symmetric and alternating representations of $\mathfrak{sl}_m$, as proved in \cite{AKS}. 

We briefly summarize the proof. A typical method to establish cyclic sieving involves two steps. Firstly, one defines a vector space $W$ with a basis labelled by $X$ and a linear automorphism $\overline{c}$ that permutes the basis in agreement with the action of $c$ on $X$. Secondly, one computes the traces of the diagonalizable operators $\overline{c}^d$ for $d\geq 0$, which give the sizes of the fixed-point sets $X^{c^d}$. 

Fix $m,\delta,\gamma$ and the corresponding sequence $\vec\lambda$. Let $V_\lambda$ be the irreducible representation of $\mathfrak sl_m$ associated to the dominant weight $\lambda$. As in \cite{FK}, we let $W$ be the space of invariants $\Inv(\vec\lambda)=\Inv(V_{\lambda^1}\otimes\cdots\otimes V_{\lambda^n})$ and $\overline{c}$ the linear operator given by rotation of tensor factors. By the Pieri rule, or more generally the Littelmann path model, the cardinality of $RT_m(\delta,\gamma)$ is equal to the dimension of this invariant space. We show that the Satake basis is labelled by $\delta$-semistandard tableaux of content $\gamma$ and that rotation of tensor factors permutes this basis according to $\delta$-promotion (up to global sign). We do so using Knutson--Tao hives, or equivalently Fock--Goncharov coordinates, together with a result of Goncharov--Shen \cite{GS}. That the Satake basis is indexed by $\delta$-semistandard tableaux for the case $\delta=(\bf v,\cdots,v)$ was shown in \cite{FKK} by more direct methods. The recent preprint \cite{BGL20} also shows that the Satake basis for arbitrary dominant weights is permuted by rotation.

The rotation action can be diagonalized by considering the fusion product to show that the trace of $\overline{c}^d$ is equal to $K_{\mu,\vec \lambda}(\zeta^d)$, precisely as in \cite[\S4.2]{FK} and \cite[\S8.3]{Wes}. Combining these two results establishes the cyclic sieving phenomenon. We observe that the proofs go through for an arbitrary sequence of dominant weights, yielding a cyclic sieving phenomenon on hives, but we are not aware of an effective method for computing the corresponding graded multiplicity polynomial defined by the fusion product.

In \S2 we give the main definitions of $\delta$-promotion and the generalized Bender--Knuth involution. In \S3 we prove Theorem \ref{thm:periodicity} and in \S4 we prove Theorem \ref{thm:csp}. Below are two examples where we let $\sigma_i$ and $\omega_i$ denote the $\mathfrak{sl}_m$ weights corresponding to the $i$th symmetric and alternating representations.

\begin{Ex}
Let $\delta=\bf(h,v,h,v,h,v)$ and $\gamma=(2,2,2,2,2,2)$, so that $\vec\lambda=(\sigma_2,\omega_2)^3$, $r=2$, $l=3$, $C_3=\langle c\rangle=\langle\partial_{R(\delta)}\circ\partial_\delta\rangle$ is the cyclic group of order $3$, $\zeta$ is a primitive $3$rd root of unity, and $\mu=(4,4,4)$. Then $\Inv((V_{\sigma_2}\otimes V_{\omega_2})^{\otimes 3})$ is $6$ dimensional. The $c$-orbits of the corresponding $\delta$-semistandard tableaux of shape $(3\times 4)$ are shown below. There's one orbit of size $3$ and three orbits of size $1$.
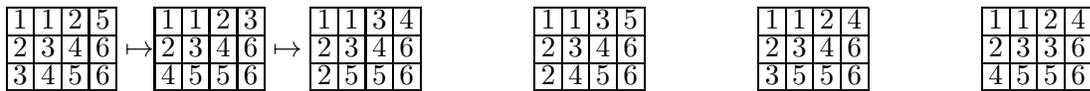
\ytableausetup{boxsize=3.5mm}
\ytableausetup{centertableaux}
\begin{align*}
\begin{ytableau}1&1&2&5\\2&3&4&6\\3&4&5&6\end{ytableau}\mapsto&\begin{ytableau}1&1&2&3\\2&3&4&6\\4&5&5&6\end{ytableau}\mapsto\begin{ytableau}1&1&3&4\\2&3&4&6\\2&5&5&6\end{ytableau}
\hspace{15mm}\begin{ytableau}1&1&3&5\\2&3&4&6\\2&4&5&6\end{ytableau}
\hspace{15mm}\begin{ytableau}1&1&2&4\\2&3&4&6\\3&5&5&6\end{ytableau}
\hspace{15mm}\begin{ytableau}1&1&2&4\\2&3&3&6\\4&5&5&6\end{ytableau}
\end{align*}
The generalized Kostka polynomial is $K_{\mu,\vec{\lambda}}(q)=q^6+q^8+2q^9+q^{10}+q^{12}$. Then $\langle\lvert\vec{\lambda}\rvert,\rho\rangle=9$, so for $f(q)=q^{9}K_{\mu,\vec{\lambda}}$ one gets $f(1)=6,f(\zeta)=3,f(\zeta^2)=3$, as desired. Below we show the jeu-de-taquin steps for a single application of $\delta$-promotion of the first tableau, resulting in a $R(\delta)$-semistandard tableau.
\begin{align*}
\begin{ytableau}\cdot&\cdot&2&5\\2&3&4&6\\3&4&5&6\end{ytableau}\mapsto&\begin{ytableau}\cdot&2&\cdot&5\\2&3&4&6\\3&4&5&6\end{ytableau}\mapsto\begin{ytableau}\cdot&2&4&5\\2&3&\cdot&6\\3&4&5&6\end{ytableau}\mapsto\begin{ytableau}\cdot&2&4&5\\2&3&5&6\\3&4&\cdot&6\end{ytableau}\mapsto\begin{ytableau}\cdot&2&4&5\\2&3&5&6\\3&4&6&\cdot\end{ytableau}\\
\mapsto&\begin{ytableau}2&\cdot&4&5\\2&3&5&6\\3&4&6&\cdot\end{ytableau}\mapsto\begin{ytableau}2&3&4&5\\2&\cdot&5&6\\3&4&6&\cdot\end{ytableau}\mapsto\begin{ytableau}2&3&4&5\\2&4&5&6\\3&\cdot&6&\cdot\end{ytableau}\mapsto\begin{ytableau}2&3&4&5\\2&4&5&6\\3&6&\cdot&\cdot\end{ytableau}\mapsto\begin{ytableau}1&2&3&4\\1&3&4&5\\2&5&6&6\end{ytableau}
\end{align*}
\end{Ex}

\begin{Ex}
Let $\delta=\bf(h,v,h,v,h,v,h,v)$ and $\gamma=(2,3,2,3,2,3,2,3)$. There are $24$ $\delta$-semistandard $(4\times 5)$-tableaux of content $\gamma$ on which $C_4=\langle \partial_{R(\delta)}\circ\partial_\delta\rangle$ acts. There are $4$ orbits of size $4$, $2$ orbits of size $2$, and $4$ orbits of size $1$. Computing via rigged configurations or the cyclage poset (in \cite{SageMath} or otherwise), the generalized Kostka polynomial is
\begin{align*}
K_{(5,5,5,5),\vec\lambda}(q)=q^{10}+q^{12}+2q^{13}+4q^{14}+2q^{15}+4q^{16}+2q^{17}+4q^{18}+2q^{19}+q^{20}+q^{22},
\end{align*}
so the cyclic sieving polynomial is $f(q)=q^{18}K_{(5,5,5,5),\vec\lambda}(q)$. Plugging in powers of a primitive $4$th root of unity $\zeta$, one gets $f(1)=24$, $f(\zeta)=4$, $f(\zeta^2)=8$, $f(\zeta^3)=4$, as desired. The corresponding $24$-dimensional invariant space of $\mathfrak{sl}_4$ is $\Inv((V_{\sigma_2}\otimes V_{\omega_3})^{\otimes 4})$.
\end{Ex}

\begin{Rem}\label{rem:Westbury}
With some additional discussion, Theorems 1 and 2 should follow from the work of Westbury \cite{Wes,WesLoc} when combined with that of Lenart \cite{Len} and van Leeuwen \cite{vLee}. This alternate proof is in the setting of crystals.

Recall that the Littelmann path model gives a model for crystals. As already alluded to, $\delta$-tableaux can be viewed as encoding dominant Littelmann paths for a tensor product of symmetric and alternating representations of $\mathfrak{sl}_m$. In particular, the rectangular $\delta$-semistandard tableaux correspond to the isolated crystal elements. The category of $\mathfrak{g}$-crystals has a coboundary structure given by the crystal commutor \cite{HK}, and one can define ``crystal promotion'' on a tensor product of crystals in terms of the commutor (see \cite[\S6.3-4]{Wes} and \cite[\S2.3]{FK}). Westbury showed in \cite{Wes} that the crystal promotion operator has order $n$ on the isolated crystal elements.

One then needs to show that this action of the crystal promotion operator on the invariant crystal elements, when encoded in terms of tableaux, agrees with the $\delta$-promotion operator defined here. One way to do this is to decompose the commutor in terms of local rules, as described in \cite{WesLoc}. Lenart \cite{Len} showed that the rule appearing in \cite{vLee} is the local rule for a tensor product of minuscule representations. This local rule when interpreted in terms of row-strict tableaux in type A is the (transpose) Bender--Knuth involution. The generalized Bender--Knuth involution given here should be the local rule for a tensor product of symmetric and alternating representations of $\mathfrak{sl}_m$.
\end{Rem}

\subsection{Acknowledgements}
The first author would like to thank Pavlo Pylyavskyy and David Speyer for independently suggesting to incorporate rows to the methods of \cite{Akh}. The authors thank Joel Kamnitzer, Allen Knutson, and Bruce Westbury for helpful comments and discussions.

\section{Definitions and Generalizations to $\delta$-semistandard Tableaux}
\subsection{$\delta$-semistandard Tableaux}
We use the English notation for partitions and tableaux, so that a Young diagram is a left-justified collection of boxes, or cells, with weakly decreasing row lengths from top down. Define an \emph{orientation string} to be a sequence $\delta=(\delta_1,\ldots,\delta_n)$ such that for all $i$ either $\delta_i=\bf h$ or $\delta_i=\bf v$, where the $\bf h$ and $\bf v$ stand for horizontal and vertical respectively. A \emph{horizontal strip} in a tableau is a subset of cells such that no two cells are contained in the same column. Similarly, a \emph{vertical strip} is a subset of cells no two of which are contained in the same row. 
\begin{Def}
Let $\delta$ be an orientation string. A tableau $T$ is a \emph{$\delta$-semistandard tableau} if the entries are weakly increasing along the rows and down columns, and the cells of $T$ that contain an $i$ form a horizontal strip if $\delta_i=\bf h$ and a vertical strip if $\delta_i=\bf v$. 
\end{Def}
The $(\bf h,\ldots,h)$-semistandard tableaux are the usual semistandard tableaux whose entries are weakly increasing along rows and strictly down columns. Similarly, the $(\bf v,\ldots,v)$-semistandard tableaux are row-strict semistandard tableaux whose entries weakly increase down columns and strictly increase along rows. The \emph{content} of a tableau $T$ is the string $\gamma=(\gamma_1,\ldots,\gamma_n)$ where $\gamma_i$ is equal to the number of cells with entry $i$ in $T$. Let $\lvert\gamma\rvert$ denote the sum $\sum \gamma_i$. The pair $(\delta,\gamma)$ can be thought of together as the orientation-content pair. Let $RT_m(\delta,\gamma)$ denote the set of $\left(m\times \frac{\lvert\gamma\rvert}{m}\right)$-rectangular $\delta$-semistandard tableaux with content $\gamma$. 

\subsection{Promotion}
To briefly recall promotion, let $T$ be a $(\bf h,\ldots,h)$-semistandard tableau. Replace every $1$ entry by a dot. Choose the rightmost dot and apply outwards jeu-de-taquin slides as follows. Let $a$ and $b$ denote respectively the entries of the cell directly below and directly to the right of the dotted cell. If $b>a$, then slide the dot one cell to the right and fill the cell where the dot used to be with the entry $b$. If $a\leq b$, then slide the dot one cell down and fill the previously dotted cell with the entry $a$. If there is no cell directly below (resp. to the right), then perform a rightward (resp. downward) slide. If $T$ does not contain an undotted cell directly to the right nor below the current dotted cell, then this completes the jeu-de-taquin slides for this dot. 

Repeat this procedure for each of the dots, applying the above jeu-de-taquin slides to the rightmost dot each time. When all dots are in the bottom-right corner of the tableau replace each dot with $n+1$ and subtract $1$ from each entry, resulting to get the promotion tableau $\partial(T)$. The result is again a semistandard tableau since the jeu-de-taquin sliding rules preserve the semistandard condition. The content of $\partial(T)$ is $R(\gamma)=(\gamma_2,\ldots,\gamma_n,\gamma_1)$.

Note that our definition of promotion agrees with the definition in \cite{Sta} and our promotion is the inverse of the promotion operator defined in in \cite{Rho}.

We now alter the above definition of jeu-de-taquin to preserve the $\delta$-semistandard condition for an arbitrary orientation string $\delta$.

\begin{Def}
Fix an orientation string $\delta$ and let $T$ be a $\delta$-semistandard tableau. The $\delta$-\emph{jeu-de-taquin} is defined as follows. Consider a dotted cell in $T$ with entry $a$ in the cell directly below and $b$ in the cell directly to the right. If $a<b$ or $a>b$, then apply the usual jeu-de-taquin rule by sliding the dot into the cell containing the smaller value of $a$ and $b$. If $a=b$, then exchange the dot with the entry of the cell directly below if $\delta_a=\bf h$ and with the cell to the right $b$ if $\delta_a=\bf v$.

To define $\delta$-\emph{promotion} on $T$ replace all of the $1$ entries with dots. If $\delta_1=\bf h$ then this will be a row of dots, and if $\delta=\bf v$ then this will be a column of dots. Perform $\delta$-jeu-de-taquin slides to the dots, starting with the rightmost dot (resp. lowest dot) if $\delta_1=\bf h$ (resp. $\delta_1=\bf v$). Once all of the dots are in a row or column in the bottom-rightmost portion of the tableau, replace each dot with $n+1$ and subtract $1$ from each entry to get the $\delta$-promotion tableau $\partial_\delta(T)$.
\end{Def}

\begin{Ex}
Let $\delta=(\bf h,v,h,v,h,v,h,v)$, $\gamma=(2,3,2,3,2,3,2,3)$, so that the following tableau $T$ is rectangular $\delta$-semistandard with content $\gamma$. Its promotion tableau $\partial_\delta(T)$ has orientation and content pair $R(\delta)=(\bf v,h,v,h,v,h,v,h)$, $R(\gamma)=(3,2,3,2,3,2,3,2)$. Note that $T$ can also be interpreted as a $\delta'$-semistandard tableau where $\delta=(\bf h,v,h,h,h,v,h,v)$, in which case $\delta'$-promotion gives a different tableau with a different orientation string $R(\delta')=(\bf v,h,h,h,v,h,v,h)$.
\begin{align*}
\ytableausetup{boxsize=3.5mm}
T=\begin{ytableau}1&1&2&3&4\\2&3&4&6&8\\2&4&5&6&8\\5&6&7&7&8\end{ytableau} \hspace{20mm} \partial_\delta(T)=\begin{ytableau}1&2&2&3&7\\1&3&4&5&7\\1&3&5&6&7\\4&5&6&8&8\end{ytableau}\hspace{20mm} \partial_{\delta'}(T)=\begin{ytableau}1&2&2&3&7\\1&3&3&5&7\\1&4&5&6&7\\4&5&6&8&8\end{ytableau}
\end{align*}
\end{Ex}

\subsection{BK Involutions}
To prove our main theorem regarding the periodicity of $\delta$-promotion on rectangular tableaux we will need the Bender--Knuth involutions.

Recall that if $T$ is semistandard in the usual sense, then promotion can be decomposed into a sequence of Bender--Knuth involutions. The $i$th BK involution $t_i$ applied to a semistandard tableau $T$ only affects the entries containing an $i$ and $i+1$ and is defined as follows (see \cite{BK, Ste}). Let $S$ be the skew tableau consisting of the cells of $T$ that contain an entry equal to $i$ or $i+1$. An entry $i$ (resp. $i+1$) of $S$ is called \emph{free} if there is no $i+1$ (resp. $i$) in the same column. In any row of $S$ the free $i$ entries appear in consecutive columns followed by free $i+1$ entries in consecutive columns. The tableau $t_i(T)$ is obtained from $T$ by exchanging the number of free $i$ entries and free $i+1$ entries in each row, so that if a row of $T$ has $a$ free $i$'s and $b$ free $i+1$'s, then the same row of $t_i(T)$ has $b$ free $i+1$'s and $a$ free $i$'s. Hence, if $T$ has content $(\gamma_1,\ldots,\gamma_n)$, then $t_i(T)$ has content $\tau_i(\gamma)=(\gamma_1,\ldots,\gamma_{i+1},\gamma_i,\ldots,\gamma_n)$. Promotion can be decomposed into the sequence $\partial(T)=t_{n-1}\cdots t_1(T)$ \cite{Gan}. See Figure \ref{fig:BK example} for an example.

We now define a modified BK involution on $\delta$-semistandard tableaux. As before, the $i$th BK involution will only affect the entries of a tableau that contain $i$ and $i+1$. In addition to transposing the $i$ and $i+1$ entries of the content string, it also transposes the $i$ and $i+1$ entries of the orientation string $\delta$. If $\delta_i$ and $\delta_{i+1}$ differ, then the skew tableau $S$ consisting of cells that contain an $i$ or $i+1$ is a disjoint union of connected ribbons. That is, $S$ does not contain a $2\times 2$ subdiagram.

\begin{Def}\label{def:mixed BK}
The $i$th $\delta$-BK involution $t_i$ applied to a $\delta$-semistandard tableau $T$ is defined as follows. If $\delta_i={\bf h}=\delta_{i+1}$, then apply the classical BK involution to get $t_i(T)$. If $\delta_i={\bf v}=\delta_{i+1}$, then transpose $T$ to get $T^t$, apply the classical BK involution to $T^t$, and transpose the result to get $t_i(T)$. 

If $\delta_i\not=\delta_{i+1}$, then let $S$ denote the ribbon tableau consisting of the cells of $T$ with an entry equal to $i$ or $i+1$. Let $C$ be a connected component of $S$ and suppose that $\delta_i={\bf h}$ and $\delta_{i+1}={\bf v}$. Let $a$ denote the top--rightmost cell of $C$ and $b$ the bottom-leftmost cell. Move every $i$ entry one cell to the right and every $i+1$ entry up one cell. The entry in cell $a$ is placed in cell $b$. Change every $i$ to an $i+1$ and every $i+1$ to an $i$. Then $t_i(T)$ is the result of applying this procedure to every connected component of $S$. Analogously, if $\delta_i={\bf v}$ and $\delta_{i+1}={\bf h}$, then in a connected component of $S$ all of the $i$ entries are moved down a cell, all of the $i+1$ entries are moved one cell to the left, the entry in cell $b$ is placed in cell $a$, and the $i$, $i+1$ values are interchanged. 

The resulting tableau $t_i(T)$ is $\tau_i(\delta)$-semistandard. When $\delta_i\not=\delta_{i+1}$ we call $t_i$ the \emph{mixed BK involution} or the \emph{mixed BK move}. Below is an example for $i=7$ where only the cells containing a $7$ or $8$ of the tableaux are shown.
\end{Def}

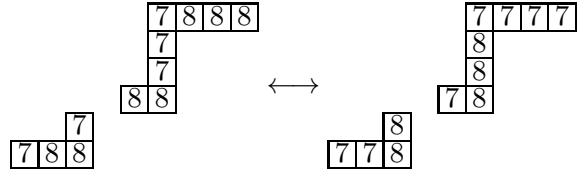
\begin{figure}
\begin{align*}
\ytableausetup{boxsize=3.5mm}
\begin{array}{p{2cm}cccccc}
&T&t_1(T)&t_2t_1(T)& t_3t_2t_1(T)&t_4t_3t_2t_1(T)\\[5mm]
tableau:&\begin{ytableau}1&1&1&2&2\\2&2&3&3&4\\3&4&4&5&5\end{ytableau}& \begin{ytableau}1&1&1&1&2\\2&2&3&3&4\\3&4&4&5&5\end{ytableau}&\begin{ytableau}1&1&1&1&3\\2&2&2&3&4\\3&4&4&5&5\end{ytableau}&\begin{ytableau}1&1&1&1&3\\2&2&2&4&4\\3&3&4&5&5\end{ytableau}&\begin{ytableau}1&1&1&1&3\\2&2&2&4&4\\3&3&5&5&5\end{ytableau}&\\[8mm]
partial promotion:&\begin{ytableau}\cdot&\cdot&\cdot&2&2\\2&2&3&3&4\\3&4&4&5&5\end{ytableau}& \begin{ytableau}1&1&1&1&\cdot\\\cdot&\cdot&3&3&4\\3&4&4&5&5\end{ytableau}&\begin{ytableau}1&1&1&1&\cdot\\2&2&2&\cdot&4\\\cdot&4&4&5&5\end{ytableau}&\begin{ytableau}1&1&1&1&3\\2&2&2&\cdot&\cdot\\3&3&\cdot&5&5\end{ytableau}&\begin{ytableau}1&1&1&1&3\\2&2&2&4&4\\3&3&\cdot&\cdot&\cdot\end{ytableau}&\\
\end{array}
\end{align*}
\caption{A semistandard tableau $T$ and applications of BK involutions. \label{fig:BK example}}
\end{figure}

\begin{align*}
\ytableausetup{boxsize=3.5mm}
\begin{ytableau}\none&\none&\none&\none&\none&7&8&8&8\\\none&\none&\none&\none&\none&7\\\none&\none&\none&\none&\none&7\\\none&\none&\none&\none&8&8\\\none&\none&7\\7&8&8\end{ytableau}
\longleftrightarrow
\begin{ytableau}\none&\none&\none&\none&\none&7&7&7&7\\\none&\none&\none&\none&\none&8\\\none&\none&\none&\none&\none&8\\\none&\none&\none&\none&7&8\\\none&\none&8\\7&7&8\end{ytableau}
\end{align*}

As mentioned above, promotion in the semistandard case can be decomposed as $\partial(T)=t_{n-1}\cdots t_1(T)$. More generally, $t_{k-1}\cdots t_1(T)$ is the result of applying the promotion operator to the tableau consisting of the cells with entries from the set $\{1,\ldots, k\}$ of $T$, while keeping all other entries fixed. Put differently, $t_{k-1}\cdots t_1(T)$ is the result of replacing the $1$ entries with dots, jeu-de-taquin sliding the dots only past the entries at most $k$, subtracting $1$ from all such entries, and replacing the dots with $k$'s. The analogous result holds for $\delta$-promotion on $\delta$-semistandard tableaux. In particular, the mixed BK move agrees with sliding dots past entries $i+1$ when the dots and $i+1$ entries form strips of differing orientations.

\begin{Prop}
Let $\delta$ be an orientation string such that $\delta_i=\bf h$, $\delta_{i+1}=\bf v$ and let $C$ be a connected component of the ribbon consisting of the $i$ and $i+1$ entries of a $\delta$-semistandard tableau $T$. The following agrees with the mixed BK move $t_i$ of Definition \ref{def:mixed BK}. Change the $i$ entries to dots, jeu-de-taquin slide the dots past the $i+1$'s starting with the rightmost dot, replace the dots with $i$'s, and interchange $i$'s and $i+1$'s. The analogous result holds for $\delta_i=\bf v$ and $\delta_{i+1}=\bf h$ where jeu-de-taquin is applied to the bottom-most dot first.
\end{Prop}
\begin{proof}
Consider the $\delta_i=\bf h$ and $\delta_{i+1}=\bf v$ case. This is almost immediate from the definitions. If there are no $i$'s, or if there are no $i+1$'s, then this agrees with Definition \ref{def:mixed BK}. Suppose $C$ contains both $i$'s and $i+1$'s, and replace all $i$'s with dots. If the top-rightmost cell contained an $i$, then the leftmost dot in the top row of $C$ will slide all the way down its column. The same holds for any row containing dots as long as the first dot has an $i+1$ entry directly below it. If the lowest row of the connected component contains dots, then these dots don't move, and in any case the lowest row contains at least one dot after sliding. This agrees with the mixed BK move since the top-rightmost entry is placed in the bottom-leftmost cell and then $i$'s and $i+1$'s interchanged. Similarly, if the top-rightmost cell contains an $i+1$, then the dot in each row slides to the right. The other case is similar.
\end{proof}

\begin{Cor}\label{cor:BK decomp}
Let $T$ be a $\delta$-semistandard tableau. Then $t_{k-1}\cdots t_1(T)=\partial_\delta\vert_{[k]}(T)$ for all $k$. In particular, $\partial_\delta(T)=t_{n-1}\cdots t_1(T)$.
\end{Cor}

\section{Proof of Order $n$}\label{sec:periodicity}
Theorem \ref{thm:periodicity} states that promotion on rectangular $\delta$-semistandard tableaux has period $n$. We prove it in this section using Knutson--Tao hives. The main idea is already present in \cite{Akh} with the main addition being the analysis of the ``mixed'' setup corresponding to the mixed BK involution.

\begin{figure}[b]
\centering
\scalemath{.5}{
\begin{tikzpicture}
  \begin{ternaryaxis}[
  ylabel=$\scalemath{2}{\sigma_3=(3,0,0,0)}$,xlabel=$\scalemath{2}{(2,1,0,0)}$,zlabel=$\scalemath{2}{(4,2,0,0)}$,
    clip=false,xticklabels={},yticklabels={},zticklabels={},
    xmin=0,
    ymin=0,
    zmin=0,
    xmax=8,
    ymax=8,
    zmax=8,
    xtick={2,4,6},
    ytick={2,4,6},
    ztick={2,4,6},
    major tick length=0,
]    
  	\node at (axis cs:0,24,0) {$\scalemath{2}{0}$};
	\node at (axis cs:2,22,0) {$\scalemath{2}{3}$};
	\node at (axis cs:4,20,0) {$\scalemath{2}{3}$};
	\node at (axis cs:6,18,0) {$\scalemath{2}{3}$};
	\node at (axis cs:8,16,0) {$\scalemath{2}{3}$};
	
	\node at (axis cs:0,22,2) {$\scalemath{2}{4}$};
	\node at (axis cs:2,20,2) {$\scalemath{2}{5}$};
	\node at (axis cs:4,18,2) {$\scalemath{2}{5}$};
	\node at (axis cs:6,16,2) {$\scalemath{2}{5}$};
	
	\node at (axis cs:0,20,4) {$\scalemath{2}{6}$};
	\node at (axis cs:2,18,4) {$\scalemath{2}{6}$};
	\node at (axis cs:4,16,4) {$\scalemath{2}{6}$};
	
	\node at (axis cs:0,18,6) {$\scalemath{2}{6}$};
	\node at (axis cs:2,16,6) {$\scalemath{2}{6}$};
	
	\node at (axis cs:0,16,8) {$\scalemath{2}{6}$};
		
  \end{ternaryaxis}
\end{tikzpicture}
}
\hspace{5mm}
\scalemath{.5}{
\begin{tikzpicture}
  \begin{ternaryaxis}[
  ylabel=$\scalemath{2}{\omega_2=(1,1,0,0)}$,xlabel=$\scalemath{2}{(2,1,0,0)}$,zlabel=$\scalemath{2}{(3,1,1,0)}$,
    clip=false,xticklabels={},yticklabels={},zticklabels={},
    xmin=0,
    ymin=0,
    zmin=0,
    xmax=8,
    ymax=8,
    zmax=8,
    xtick={2,4,6},
    ytick={2,4,6},
    ztick={2,4,6},
    major tick length=0,
]    
  	\node at (axis cs:0,24,0) {$\scalemath{2}{0}$};
	\node at (axis cs:2,22,0) {$\scalemath{2}{1}$};
	\node at (axis cs:4,20,0) {$\scalemath{2}{2}$};
	\node at (axis cs:6,18,0) {$\scalemath{2}{2}$};
	\node at (axis cs:8,16,0) {$\scalemath{2}{2}$};
	
	\node at (axis cs:0,22,2) {$\scalemath{2}{3}$};
	\node at (axis cs:2,20,2) {$\scalemath{2}{4}$};
	\node at (axis cs:4,18,2) {$\scalemath{2}{4}$};
	\node at (axis cs:6,16,2) {$\scalemath{2}{4}$};
	
	\node at (axis cs:0,20,4) {$\scalemath{2}{4}$};
	\node at (axis cs:2,18,4) {$\scalemath{2}{5}$};
	\node at (axis cs:4,16,4) {$\scalemath{2}{5}$};
	
	\node at (axis cs:0,18,6) {$\scalemath{2}{5}$};
	\node at (axis cs:2,16,6) {$\scalemath{2}{5}$};
	
	\node at (axis cs:0,16,8) {$\scalemath{2}{5}$};
	
	\addplot3[color=red] coordinates {(4,4,0)(2,4,2)(2,2,4)(0,2,6)(0,0,8)};
		
  \end{ternaryaxis}
\end{tikzpicture}
}
\caption{Two examples of $3$-hives with $m=4$. The break path of the hive on the right is indicated in red. \label{fig:hive examples}}
\end{figure}
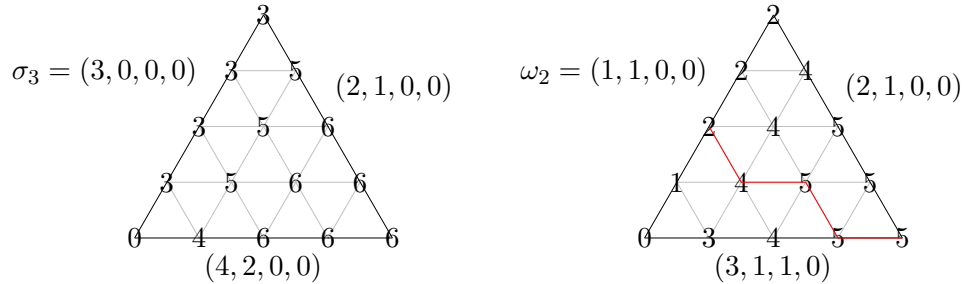

\subsection{Knutson--Tao Hives}
We briefly recall Knutson--Tao hives \cite{KT}. Hives count $GL_m$ tensor product multiplicities. Let 
\begin{align*}
\Delta_m^3=\{(i,j,k)\in \mathbb Z_{\geq0}^3\mid i+j+k=m\}
\end{align*}
be the set of lattice points in a triangle of size $m$. A \emph{hive}, or \emph{3-hive}, is an assignment of an integer $f_{i,j,k}$ to each point of $\Delta_m^3$ that satisfies the following rhombus inequalities, 
\begin{align*}
f_{i,j,k}+f_{i,j+1,k-1}&\geq f_{i+1,j,k-1}+f_{i-1,j+1,k}\\
f_{i,j,k}+f_{i+1,j-1,k}&\geq f_{i+1,j,k-1}+f_{i,j-1,k+1}\\
f_{i,j,k}+f_{i+1,j,k-1}&\geq f_{i,j+1,k-1}+f_{i+1,j-1,k},
\end{align*}
for all $(i,j,k)$ in $\Delta_m^3$. The rhombus inequalities state that for any unit rhombus in $\Delta_m^3$ the sum of the labels of the two vertices across the short diagonal is at least the sum of the labels of the vertices across the long diagonal. Two hives are considered equivalent if there is a constant $c$ such that their labels at each lattice point differ by $c$. See Figure \ref{fig:hive examples} for examples. 

As a consequence of the hive inequalities, the successive differences of the edge values $(f_{m,0,0},f_{m-1,1,0},\ldots,f_{0,m,0})$ give a weakly decreasing sequence of integers 
\begin{align*}
\lambda=(f_{m-1,1,0}-f_{m,0,0},f_{m-2,2,0}-f_{m-1,2,0},\ldots,f_{1,m-1,0}-f_{0,m,0}).
\end{align*}
Similarly, the successive differences of the hive values along the other two external edges $(f_{0,m,0},f_{0,m-1,1},\ldots,f_{0,0,m})$ and $(f_{m,0,0},f_{m-1,0,1},\ldots,f_{0,0,m})$ also give weakly decreasing sequences $\mu$ and $\nu$. In this case, we will say that the $3$-hive is of \emph{type} $(\lambda,\mu;\nu)$. Recall that the dominant weights of $GL_m$ are weakly decreasing sequences of integers $\lambda_1\geq \lambda_2\geq \cdots \geq\lambda_m$. Reading the differences along an edge in the opposite order gives the dual dominant weight $\lambda^*=(-\lambda_m,\ldots,-\lambda_1)$. When a dominant weight has no negative parts we will often interpret it as a partition.

We will also need the definition of $n$-hives, although we will primarily be working with $4$-hives. Let
\begin{align*}
\Delta_m^n=\{(i_1,\ldots,i_n)\in \mathbb Z_{\geq 0}^n\mid i_1+\cdots +i_n=m\}
\end{align*}
be the set of lattice points in an $(n-1)$-dimensional tetrahedron of size $m$. An $n$-hive is an assignment of integers $f_{i_1,\ldots,i_n}$ to each point of $\Delta_m^n$ such that the rhombus inequalities are satisfied when restricting to every two-dimensional face, and for all $\vec{j}=(j_1,\ldots,j_n)$ such that $\sum j_k=m-2$ and all $1\leq a<b<c<d\leq n$ the following octahedron recurrence holds.
\begin{align}\label{oct rec}
f_{\vec{\jmath}+e_a+e_c}+f_{\vec{\jmath}+e_b+e_d}=\max\left(f_{\vec{\jmath}+e_a+e_b}+f_{\vec{\jmath}+e_c+e_d},f_{\vec{\jmath}+e_a+e_d}+f_{\vec{\jmath}+e_b+e_c}\right)
\end{align}
Here $e_k$ denotes the vector with a $1$ in the $k$th position and $0$'s elsewhere. As mentioned, the most important case for the proof of periodicity will be when $n=4$, so that the hive values label lattice points of a tetrahedron. In this case, the octahedron recurrence can be visualized as shown in Figure \ref{fig:octahedron}. For each $i$ in the range $1\leq i \leq n-1$ let $\lambda^i$ be the dominant weight given by the successive differences of the $n$-hive along the edge from lattice point $m\cdot e_i$ to $m\cdot e_{i+1}$ and $\mu$ given by the differences along the edge from $m\cdot e_1$ to $m\cdot e_n$. The \emph{type} of the $n$-hive is then defined to be $(\lambda^1,\ldots,\lambda^{n-1};\mu)$.
\begin{figure}[b]
\begin{center}
\begin{tikzpicture}[line join=bevel,z=-7]
\coordinate (A1) at (0,0,-1);
\coordinate (A2) at (-1,0,0);
\coordinate (A3) at (0,0,1);
\coordinate (A4) at (1,0,0);
\coordinate (B) at (0,1,0);
\coordinate (C) at (0,-1,0);

\node at (A1) {b};
\node[left] at (A2) {a};
\node[below,left] at (A3) {d};
\node[right] at (A4) {c};
\node[above] at (B) {e};
\node[below] at (C) {f=max(a+c,b+d)-e};
\draw[dotted] (A1) -- (A2) -- (B) -- cycle;
\draw[dotted] (A4) -- (A1) -- (B) -- cycle;
\draw[dotted] (A1) -- (A2) -- (C) -- cycle;
\draw[dotted] (A4) -- (A1) -- (C) -- cycle;
\draw (A2) -- (A3) -- (B) -- cycle;
\draw (A3) -- (A4) -- (B) -- cycle;
\draw (A2) -- (A3) -- (C) -- cycle;
\draw (A3) -- (A4) -- (C) -- cycle;
\end{tikzpicture}
\caption{The octahedron recurrence. \label{fig:octahedron}}
\end{center}
\end{figure}
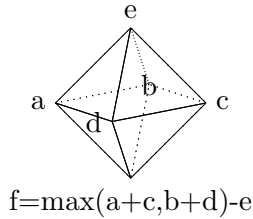

Knutson and Tao \cite{KT} established the following theorem for $n=3,4$, which was later extended to all $n$ in \cite{Kam, GS}. For a dominant weight $\lambda$ let $V_{\lambda}$ denote the corresponding irreducible representation of $GL_m$.
\begin{Thm}[\cite{KT,Kam,GS}]\label{thm:hives}
Let $(\lambda^1,\ldots,\lambda^{n-1};\mu)$ be a sequence of dominant weights. The number of $n$-hives of type $(\lambda^1,\ldots,\lambda^{n-1};\mu)$ is equal to the multiplicity of $V_{\mu}$ in $V_{\lambda^1}\otimes\cdots\otimes V_{\lambda^{n-1}}$.
\end{Thm}

Let $\omega_i=(1,\ldots,1,0,\ldots,0)$ denote the $i$th fundamental weight where there are $i$ many $1$'s. Let $\sigma_i$ denote the weight $(i,0,\ldots,0)$ corresponding to the $i$th symmetric power representation. Recall that for any irreducible representation $V_\lambda$, the decompositions of $V_{\omega_i}\otimes V_\lambda$ and $V_{\sigma_i}\otimes V_\lambda$ into direct sums are multiplicity free. By the previous theorem the number of $3$-hives of type $(\omega_i,\lambda;\mu)$ and $(\sigma_i,\lambda;\mu)$ is either $0$ or $1$. By the Pieri rule the irreducible representations occurring in $V_{\omega_i}\otimes V_\lambda=\oplus V_\mu$ correspond to $\mu$ gotten from $\lambda$ by adding a vertical strip of size $i$. Similarly, those arising in $V_{\sigma_i}\otimes V_\lambda=\oplus V_\mu$ correspond to $\mu$ gotten from $\lambda$ by adding a horizontal strip of size $i$. In this case, the $3$-hives are easy to understand, as described by the following two lemmas. The first appears in \cite[Lem.~2]{KT}, and the second is similar so we omit the proofs. In the following, for a $GL_m$-weight $\lambda$ we will use the notation $\lambda^{(k)}$ to denote $(\lambda_1,\ldots,\lambda_k)$, the truncation of $\lambda$ to the first $k$ parts.

\begin{Lem}[{\cite[Lem.~2]{KT}}]\label{lem:alt}
Let $f$ be a $GL_m$ $3$-hive of type $(\omega_i,\lambda;\mu)$, so that the northwest external edge has successive differences $\omega_i$ when read from southwest to northeast. Then the southwest-to-northeast oriented strip located one step in from the northwest external edge has successive differences $\omega_i^{(m-1)}$ or $\omega_{i-1}^{(m-1)}$. The first case occurs if and only if $\lambda_1=\mu_1$ and the second case if and only if $\mu_1=\lambda_1+1$. The pattern continues so that the $j$th southwest-to-northeast strip has differences given by $\omega_{i_j}^{(m-j+1)}$ for $i=i_1\geq i_2\geq\cdots\geq i_m$ such that $i_{j+1}=i_j$ if and only if $\mu_j=\lambda_j$ and $i_{j+1}=i_j-1$ if and only if $\mu_j=\lambda_j+1$. The analogous result holds for a hive of type $(\lambda,\omega_i;\mu)$.
\end{Lem}

We can encode this information in the form of a lattice path through the hive that begins at the northwest external edge and ends at the vertex $me_3$. With the notation of the previous lemma, the path visits the vertices 
\begin{align*}
(m-i_1,i_1,0), (m-i_2-1,i_2,1), (m-i_3-2,i_3,2), \ldots, (1-i_{m-1},i_{m-1},m-1),(0,0,m). 
\end{align*}
We call this the \emph{break path}. Note that the hive values along a strip increase by $1$ until they reach the break path, after which they are constant. The break path is indicated in the hive on the right of Figure \ref{fig:hive examples}. When drawn as in Figure \ref{fig:hive examples} there is a slanted step between the $j$ and $j+1$ vertices if and only if $\mu_j=\lambda_j+1$ and a horizontal step if and only if $\mu_j=\lambda_j$. 

\begin{Lem}\label{lem:sym}
Let $f$ be a $GL_m$ $3$-hive of type $(\sigma_i,\lambda;\mu)$, so that the northwest external edge has successive differences $\sigma_i$ when read from southwest to northeast. The $j$th southwest-to-northeast strip has differences $\sigma_{i_j}^{(m-j+1)}$ for some $i_j$ such that $i=i_1\geq i_2 \geq \cdots \geq i_m$. More precisely, the $i_j$ are given by the equation $\mu_{j}=\lambda_j+(i_j-i_{j+1})$. The analogous result holds for a hive of type $(\lambda,\sigma_i;\mu)$.
\end{Lem}

Let $(\lambda^1,\ldots,\lambda^n)$ be a sequence of $GL_m$-weights such that for all $i$, $\lambda^i$ is equal to $\omega_{j_i}$ or $\sigma_{j_i}$. We will call such a sequence a \emph{sym-alt sequence}. Let $\delta$ be the corresponding orientation string, so that $\delta_i=\bf v$ if $\lambda^i=\omega_{j_i}$ and $\delta_i=\bf h$ if $\lambda^i=\sigma_{j_i}$. Let $s=\sum j_i$ and $\mu$ be the $m\times (s/m)$ rectangular partition. By the iterated Pieri rule the multiplicity of the the irreducible representation $V_\mu$ in $V_{\lambda^1}\otimes\cdots\otimes V_{\lambda^n}$ is equal to the number of $\delta$-semistandard tableaux of shape $\mu$ and content $(j_1,j_2,\ldots , j_n)$. By Theorem \ref{thm:hives} this is equal to the number of $(n+1)$-hives of type $(\lambda^1,\ldots,\lambda^n;\mu)$.

To see the bijection between $(n+1)$-hives and $\delta$-semistandard tableaux, let $\alpha_{i,j}$ denote the dominant weight given by the successive differences of the hive values at the lattice points along the edge $(me_i,me_j)$. Then each subhive with vertices $(me_1,me_i,me_{i+1})$ is a $3$-hive of type $(\alpha_{1,i},\lambda^i;\alpha_{1,i+1})$. By Theorem \ref{thm:hives} and the Pieri rule the dominant weights $\alpha_{1,i}$ and $\alpha_{1,i+1}$ can be interpreted as partitions that differ by a strip of size $j_i$ with orientation depending on $\delta_i$. Then $\alpha_{1,1},\alpha_{1,2},\ldots,\alpha_{1,n}$ is a nested sequence of partitions ending at a rectangular shape $\mu$, which can be interpreted $\delta$-semistandard tableau $T$, and this is a bijection. 

The following is the main proposition that we will need to prove periodicity. It says that the weights $\alpha_{i,j},\alpha_{i+1,j},\alpha_{i,j+1},\alpha_{i+1,j+1}$, when appropriately interpreted as skew tableaux, are related by the $\delta$-BK move. More precisely, consider sequences of partitions $\lambda \subset \mu \subset \nu$, $\lambda\subset \rho \subset \nu$ such that the skew shapes $\mu/\lambda$ and $\nu/\rho$ are horizontal strips, while $\nu/\mu$ and $\rho/\lambda$ are vertical strips, so in particular, $\nu/\lambda$ is a ribbon. The sequences $\lambda \subset \mu \subset \nu$ and $\lambda\subset \rho \subset \nu$ can be interpreted as a pair of skew tableau with $\mu/\lambda$, $\rho/\lambda$ containing $1$'s and $\nu/\mu$, $\nu/\rho$ containing $2$'s. We will say that the two sequences are related by the mixed BK move if these corresponding ribbon tableaux are related by the mixed BK move.

\begin{Prop}\label{prop:main}
Consider a $4$-hive with $\alpha_{i,j}$ being the dominant weights given by the successive differences of the hive values along the edges from $me_i$ to $me_j$ for $1\leq i\leq j\leq 4$. Suppose that $\alpha_{1,2}=\sigma_l$ and $\alpha_{3,4}=\omega_k$ and that the $\alpha_{i,j}$ for $1\leq i\leq j\leq 4$ do not have negative parts. Then $\alpha_{1,3}/\alpha_{2,3}$ and $\alpha_{1,4}/\alpha_{2,4}$ are horizontal strips, and $\alpha_{1,4}/\alpha_{1,3}$ and $\alpha_{2,4}/\alpha_{2,3}$ are vertical strips. The sequences of partitions 
\begin{align*}
\alpha_{2,3}\subset \alpha_{1,3}\subset \alpha_{1,4}\\
\alpha_{2,3}\subset \alpha_{2,4}\subset \alpha_{1,4}
\end{align*}
are related by the mixed BK move. The analogous result holds for $\alpha_{1,2}=\omega_k$ and $\alpha_{3,4}=\sigma_l$. Similarly, if $\alpha_{1,2}=\sigma_l$, $\alpha_{3,4}=\sigma_k$, then the partitions are related by the usual BK move, and if $\alpha_{1,2}=\omega_l$, $\alpha_{3,4}=\omega_k$, then the partitions are related by the transpose of the usual BK move.
\end{Prop}

Before giving the proof of Proposition \ref{prop:main} we apply it to give the proof of the main periodicity theorem. To prove periodicity we need one more easy lemma.

\begin{Lem}\label{lem:rectangle}
Let $k\omega_m=(k,\ldots,k)$. There is one hive of type $(k\omega_m,\lambda;\mu)$ if an only if $\mu_i=\lambda_i+k$ for all $i$. There are zero such hives otherwise.
\end{Lem}

\begin{proof}[Proof of Periodicity (Theorem \ref{thm:periodicity})]
Let $T$ be a $\left(m\times \frac{\lvert \gamma\rvert}{m}\right)$-rectangular $\delta$-semistandard tableau with content $\gamma$. To the pair $(\delta,\gamma)$ associate the corresponding $GL_m$ sym-alt sequence $\vec\lambda=(\lambda^1,\ldots,\lambda^n)$ as described above and let $\mu$ be the $\left(m\times \frac{\lvert \gamma\rvert}{m}\right)$-rectangular shape. Let $f_{\vec i}$ be the corresponding $(n+1)$-hive of type $(\vec\lambda;\mu)$ and $\alpha_{ij}$ the dominant weight along the $(me_i,me_j)$-edge for all $i, j$. 

Consider the following staircase-shaped set of lattice points indexed by pairs of integers 
\begin{align*}
St_{n+1}=\{(i,j)\mid i\leq j\leq i+n+1\}.
\end{align*}
The lattice point $(i,j)$ is viewed as being in the $i$th row and $j$th column as in matrix notation with indices increasing down and to the right. Label each lattice point with the weight $\alpha_{ij}$ where indices are taken mod $(n+1)$. In particular, $\alpha_{ii}=\vec 0=\alpha_{i,i+n+1}$ for all $i$, and the first line is labelled with $\alpha_{11},\ldots,\alpha_{1n}$ (encoding the tableau $T$) and ends at $\alpha_{1,n+1}=\vec0$. Also, note that $\alpha_{j,i+n+1}=\alpha_{ij}^*$ since reading successive differences in the opposite order along an edge of a hive gives the dual weight.

Here is an example diagram with $n=4$.
\begin{align*}
\begin{array}{ccccccccccc}
\alpha_{11}&\alpha_{12}&\alpha_{13}&\alpha_{14}&\alpha_{15}&\alpha_{16}\\
&\alpha_{22}&\alpha_{23}&\alpha_{24}&\alpha_{25}&\alpha_{26}&\alpha_{27}\\
&&\alpha_{33}&\alpha_{34}&\alpha_{35}&\alpha_{36}&\alpha_{37}&\alpha_{38}\\
&&&\alpha_{44}&\alpha_{45}&\alpha_{46}&\alpha_{47}&\alpha_{48}&\alpha_{49}\\
&&&&\alpha_{55}&\alpha_{56}&\alpha_{57}&\alpha_{58}&\alpha_{59}&\alpha_{510}\\
&&&&&\alpha_{66}&\alpha_{67}&\alpha_{68}&\alpha_{69}&\alpha_{610}&\alpha_{611}\\
\end{array}
\end{align*}
\begin{align*}
\begin{array}{ccccccccccc}
\vec 0&\lambda^1&\alpha_{13}&\alpha_{14}&\mu&\vec0\\
&\vec 0&\lambda^2&\alpha_{24}&\alpha_{25}&{\lambda^1}^*&\vec0\\
&&\vec 0&\lambda^3&\alpha_{35}&\alpha_{13}^*&{\lambda^2}^*&\vec0\\
&&&\vec 0&\lambda^4&\alpha_{14}^*&\alpha_{24}^*&{\lambda^3}^*&\vec0\\
&&&&\vec 0&\mu^*&\alpha_{25}^*&\alpha_{35}^*&{\lambda^4}^*&\vec0\\
&&&&&\vec 0&\lambda^1&\alpha_{13}&\alpha_{14}&\mu&\vec0\\
\end{array}
\end{align*}

Consider first the weights for the triangular set of indices 
\begin{align*}
Tr=\{(i,j)\mid 1\leq i \leq n, i\leq j\leq n+1\}. 
\end{align*}
Note that each weight $\alpha_{ij}$ for $(i,j)\in Tr$ has no negative parts, so can be interpreted as a partition. By Proposition \ref{prop:main}, the partitions $\alpha_{ij},\alpha_{ij+1},\alpha_{i+1j},\alpha_{i+1j+1}$ around each unit square in this subdiagram are related by the $\delta$-BK involution. Similarly, every weight $\alpha_{ij}$ for $(i,j)$ in the dual triangular set of indices $Tr^*=\{(i,j)\mid n+2\leq j\leq 2n+2, j\leq i\leq n+1\}$ has no positive parts, and when interpreted appropriately each unit square satisfies the $\delta$-BK involution. Since $\alpha_{j,i+n+1}=\alpha_{ij}^*$, each $\alpha_{j,i+n+1}+\mu$ is the complement of the partition $\alpha_{ij}$ in $\mu$. Modify the diagram by replacing $\alpha_{ij}$ with $\alpha_{ij}+\mu$ for all $(i,j)\in Tr^*$. By Lemma \ref{lem:rectangle}, $\alpha_{i,n+1}=\alpha_{i,n+2}+\mu$ for all valid $i$ and $\alpha_{n+2,j}=\alpha_{n+1,j}$ for all valid $j$, so after the modification the $n+1$ and $n+2$ columns are equal and the $n+1$ and $n+2$ rows are equal, so we can merge them. 

The result is a diagram consisting of partitions that are related by the BK involution around each unit square. Since $\partial(T)=t_{n-1}\cdots t_1(T)$ by Corollary \ref{cor:BK decomp}, each row encodes the promotion tableau of the previous row and periodicity follows since $\alpha_{1j}=\alpha_{1+n+2,j+n+2}$.
\end{proof}

\begin{Rem}
Alternatively, one could consider hives for $SL_m$ where the hive values are in $\frac{1}{m}\mathbb Z$. Then the resulting staircase diagram of $\alpha_{ij}$ (when properly interpreted as partitions) would be precisely the promotion diagram of $T$, without having to shift and merge as in the proof of Theorem \ref{thm:periodicity}. But then we would have to deal with fractions in the proof of Proposition \ref{prop:main}.
\end{Rem}

\begin{Rem}
The previous proof can be interpreted in terms of Fock--Goncharov cluster variables on the Satake fiber where each triangulation of the $n$-gon corresponds to a cluster coordinate system. One begins with the cluster coordinates corresponding to the fan triangulation with the values along the edges $(1,i)$ encoding the $\delta$-semistandard tableau. Then a sequence of mutations, equivalent to applications of the octahedron recurrence in the hive as in the proof of Proposition \ref{prop:main}, flips an edge to get a neighboring triangulation. Flipping each edge of the initial fan triangulation yields the neighboring fan triangulation with edges $(2,i)$ and corresponds to promotion. Rotating $n$ times gets back to the original coordinate system.
\end{Rem}

\begin{proof}[Proof of Proposition \ref{prop:main}]
We will prove the ``mixed'' case $\alpha_{1,2}=\sigma_l$ and $\alpha_{3,4}=\omega_k$, the other cases being similar (see \cite{Akh} for the $\omega_l,\omega_k$ case).

To simplify notation let $\lambda=\alpha_{2,3}, \mu=\alpha_{1,3}, \nu=\alpha_{1,4}$ and $\rho=\alpha_{2,4}$, which we think of as two sequences of nested partitions $\lambda\subset \mu\subset \nu$ and $\lambda\subset \rho\subset \nu$. By Lemmas \ref{lem:alt} and \ref{lem:sym} $\mu/\lambda$ and $\nu/\rho$ are horizontal strips and $\nu/\mu$ and $\rho/\lambda$ are verticals trips. We interpret the sequence $\lambda\subset \mu\subset \nu$ as a ribbon-shaped tableau $T$ with $1$'s in the cells of $\mu/\lambda$ and $2$'s in the cells of $\nu/\mu$. Likewise, $\lambda\subset \rho\subset \nu$ is interpreted as a ribbon-shaped tableau $T'$ with $1$'s in the cells of $\rho/\lambda$ and $2$'s in the cells of $\nu/\rho$. A connected component of $T$ can be in one of the following two forms, depending on whether the top row contains a $2$ or not:

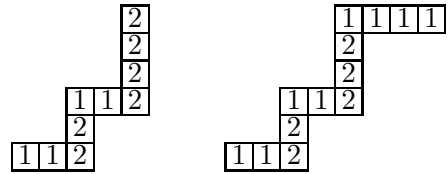
\begin{align*}
\ytableausetup{boxsize=3.5mm}
\begin{ytableau}\none&\none&\none&\none&2\\\none&\none&\none&\none&2\\\none&\none&\none&\none&2\\\none&\none&1&1&2\\\none&\none&2\\1&1&2\end{ytableau}
\hspace{10mm}
\begin{ytableau}\none&\none&\none&\none&1&1&1&1\\\none&\none&\none&\none&2\\\none&\none&\none&\none&2\\\none&\none&1&1&2\\\none&\none&2\\1&1&2\end{ytableau}
\end{align*}

The first case can be characterized as having a $2$ in every row, whereas the second case has a two in every row but the top. In the first case, we must show that the corresponding connected component of $T'$ contains a $1$ in every row. In the second case, we must show that it contains a $1$ in every row but the last. The rest is determined by the common shape $\nu$.
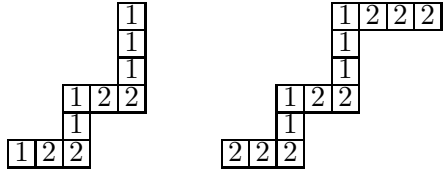
\begin{align*}
\ytableausetup{boxsize=3.5mm}
\begin{ytableau}\none&\none&\none&\none&1\\\none&\none&\none&\none&1\\\none&\none&\none&\none&1\\\none&\none&1&2&2\\\none&\none&1\\1&2&2\end{ytableau}
\hspace{10mm}
\begin{ytableau}\none&\none&\none&\none&1&2&2&2\\\none&\none&\none&\none&1\\\none&\none&\none&\none&1\\\none&\none&1&2&2\\\none&\none&1\\2&2&2\end{ytableau}
\end{align*}

Orient the $4$-hive such that it is balanced on the $(me_2,me_4)$ edge oriented vertically, so that it's labelled by $\rho$ when read from top to bottom. View the $4$-hive from above, so that the horizontally oriented $(me_1,me_3)$ edge , labelled by $\mu$ when read left to right, is closest to the viewer. The $3$-hives with vertices $(me_1,me_2,me_3)$ and $(me_1,me_3,me_4)$ are glued along $(me_1,me_3)$ and are closest to the viewer (see Figure \ref{fig:4hive example}). We will call the $3$-hive with vertices $(me_1,me_2,me_3)$ the top face and the one with $(me_1,me_3,me_4)$ the bottom face.

The northwestern external strip of the top face is labelled by the weight $\alpha_{1,2}=\sigma_l$. There are $m+1$ southwest-to-northeast oriented strips in this hive with the $j$th one labelled by $\sigma_{a_j}^{(m-j)}$ such that $l=a_0\geq a_1\geq \cdots\geq a_{m}=0$. Note that the index $j$ runs from $0$ to $m$ corresponding to strips from west to east. From this it follows that the hive values along the $j$th strip increase by $a_j$ in the first step and are constant thereafter. Hence, the first $m-1$ parts of $\lambda$ are given by the successive differences along the row one above the main horizontal: $\lambda_i=f_{(m-(i+1),1,i,0)}-f_{(m-i,1,i-1,0)}$ for $1\leq i\leq m-1$. Recall also that $\mu_j=\lambda_j+(a_{j-1}-a_{j})$.

By similar reasoning applied to the $(me_1,me_2,me_4)$ face (which is obscured from view), we see that the first $m-1$ parts of $\rho$ are given by the successive differences along the northwest-to-southeast oriented strip that is one step in from the boundary: $\rho_i=f_{(m-(i+1),1,0,i)}-f_{(m-i,1,0,i-1)}$ for $1\leq i\leq m-1$. Let $z_j$ denote the differences $f_{(m-1-j,1,0,j)}-f_{(m-j,0,0,j)}$, so that $\nu_j=\rho_j+(z_{j-1}-z_{j})$. 

The southeastern external strip of the bottom face is labelled by the weight $\gamma_{3,4}=\omega_k$ when read northeast to southwest. The $j$th northeast-to-southwest oriented strip is labelled by $\omega_{k_j}$ where we index also from west to east from $0$ to $m$ to coincide with the top face. Then $k_0=0$, $k_{m}=k$, and by Lemma \ref{lem:alt} we have $k_{j}=k_{j-1}$ if and only if $\nu_j=\mu_j$ and $k_{j}=k_{j-1}+1$ if and only if $\nu_j=\mu_j+1$. 

We view the break path of this face as beginning at vertex $me_1$. The first edge of the path connects $me_1=(m,0,0,0)$ to $(m-1,0,1,0)$ if $k_1=k_0$ and to $(m-1,0,0,1)$ if $k_1=k_0+1$. If the path is currently at vertex $v$, then the next edge ends at $v+(-1,0,1,0)$ if $k_{j}=k_{j-1}$ (horizontal step) and at $v+(-1,0,0,1)$ if $k_{j}=k_{j-1}$ (slanted step). Hence, $\nu_j=\mu_j$ if strips $j-1$ and $j$ have a horizontal step between them in the break path and $\nu_j=\mu_j+1$ if there is a slanted step. The hive values along a northeast-to-southwest strip of the bottom face increase by $1$ until they reach the break path, after which they remain constant.

We will excavate each lattice point of the bottom face using the octahedron recurrence to reveal the new two-dimensional face with coordinates $\{(i_1,1,i_3,i_4)\mid i_1+i_3+i_4=m-1\}$. This $3$-hive has type $(\lambda^{(m-1)},\omega_{k'}^{(m-1)};\rho^{(m-1)})$ where $k'=k$ or $k-1$. This will reveal a new break path from which we can determine the rows of $\rho/\lambda$. In particular let $C$ be a connected component of $\nu/\lambda$. If every row of $C$ contains a cell from $\nu/\mu$ we will show that every row of $C$ contains a cell from $\rho/\lambda$, in agreement with the mixed BK move. On the other hand, if every row of $C$ except for the top row contains a cell from $\nu/\mu$, then we will show that every row of $C$ except for the bottom row contains a cell of $\rho/\lambda$, also in agreement with the BK move. To do so, we analyze the possible change in break path upon excavation. In particular, we will see that the $j$th newly revealed northeast-to-southwest oriented strip is labelled by $\omega_{k_j}^{(m-1)}$ or $\omega_{k_j+1}^{(m-1)}$ for all $j$.

First consider excavating a point on the main horizontal on the $j$th strip. The local hive values are the following.
\begin{align*}
\begin{array}{ccccccc}
&&\sum_1^{j-1}\mu_i+a_{j-1}&&\sum_1^{j}\mu_i+a_{j}&&\\
&\sum_1^{j-1}\mu_i&&\sum_1^{j}\mu_i&&\sum_1^{j+1}\mu_i&\\
&&\sum_1^{j}\mu_i+x&& \sum_1^{j+1}\mu_i+y&&
\end{array}
\end{align*}
The possible $x$ and $y$ values are $(x=0,y=0)$, $(x=0,y=1)$, and $(x=1,y=1)$. Since $\lambda\subset \mu$ differ by a horizontal strip we have $\mu_{j+1}\leq \lambda_j$. Combining with $\mu_j=\lambda_j+(a_{j-1}-a_{j})$, we have that $\mu_{j+1}\leq \mu_j-(a_{j-1}-a_{j})$. We will use the inequality $\mu_{j+1}+a_{j-1}\leq \mu_j+a_{j}$ in the calculation of the octahedron recurrence below. 

Consider first the case of $(x=1,y=1)$ which corresponds to the excavation point $(m-j,0,j,0)$ being strictly above the break path. Applying the octahedron recurrence gives the following.
\begin{align*}
&\max\left(\sum^{j-1}\mu_i+a_{j-1}+\sum^{j+1}\mu_i+1,2\left(\sum^j\mu_i\right)+a_{j}+1\right)-\sum^j\mu_i\\
=&\max\left(\sum^{j-1}\mu_i+\mu_{j+1}+a_{j-1}+1,\sum^{j-1}\mu_i+\mu_j+a_{j}+1\right)\\
=&\sum^{j}\mu_i+a_{j}+1.
\end{align*}
The new value is one greater than the value directly to the northeast at $(m-j-1,1,j,0)$. By similar reasoning, one can show that excavating any point on the bottom face that is strictly above the break path produces a similar result - the new value is one greater than the value at the (already revealed) lattice point directly to the northeast. Hence, we may assume that every point strictly above the break path has been excavated, so in particular we know that the new values along each strip have differences $\omega_{q_j}^{(m-1)}$ for some $q_j\geq k_j$. 

Now consider excavating a point on the break path. First consider excavating a point on the horizontal segment of the break path that is closest to the main horizontal edge of the $4$-hive. Suppose that this is on the $c$th horizontal down from the main horizontal (so a lattice point $(m-j-c,0,j,c)$). The local data at such a point along the $j$ strip is the following.
\begin{align*}
\begin{array}{ccccccc}
&&\sum_1^{j-1}\mu_i+a_{j-1}+c&&\sum_1^{j}\mu_i+a_{j}+c&&\\
&\sum_1^{j-1}\mu_i+c&&\sum_1^{j}\mu_i+c&&\sum_1^{j+1}\mu_i+c&\\
&&\sum_1^{j}\mu_i+c+x&& \sum_1^{j+1}\mu_i+c+y&&
\end{array}
\end{align*}

If both $(m-c-j,0,j,c)$ and $(m-c-j-1,0,j+1,c)$ are on the break path, then we have that $x=0$ and $y=0$, which gives
\begin{align*}
&\max\left(\sum^{j-1}\mu_i+a_{j-1}+\sum^{j+1}\mu_i+2c,2\left(\sum^j\mu_i\right)+a_{j}+2c\right)-\sum^j\mu_i-c\\
=&\sum^{j}\mu_i+a_{j}+c.
\end{align*}
This value is equal to the value directly to the northeast. The remaining excavated values must be equal (since we know each strip is an $\omega_i^{(m-1)}$ for some $i$), so the revealed strip has differences given by the weight $\omega_{k_{j}}^{(m-1)}$ from northeast to southwest.

Now suppose that the excavation point is at an ``out elbow'' on the break path, so that $x=0, y=1$. This is the interesting case.
\begin{align*}
&\max\left(\sum^{j-1}\mu_i+a_{j-1}+\sum^{j+1}\mu_i+1+2c,2\left(\sum^j\mu_i\right)+a_{j}+2c\right)-\sum^j\mu_i-c\\
=&\max\left(\sum^{j-1}\mu_i+\mu_{j+1}+a_{j-1}+c+1,\sum^{j-1}\mu_i+\mu_j+a_{j}+c\right)\\
\end{align*}
This is equal to $\sum^j \mu_i+a_{j}+c$ if $\mu_{j+1}+a_{j-1}<\mu_j+a_{j}$ and is equal to $\sum^j \mu_i+a_{j}+c+1$ if $\mu_{j+1}+a_{j-1}=\mu_j+a_{j}$. In the former case, the strip is labelled by $\omega_{k_{j}}^{(m-1)}$. In the latter case, the new value is one greater than the value at the lattice point to the northeast, and the remaining values along the revealed strip can be shown to be equal. Hence, the $j$th strip has differences $\omega_{k_{j}+1}^{(m-1)}$.

Now consider excavating the point on the break path directly after an out elbow. The local data is the following where $z=-1$ or $z=0$ depending on the two cases just mentioned ($z=0$ being the ``special case'').
\begin{align*}
\begin{array}{ccccccc}
&&\sum_1^{j-1}\mu_i+a_{j-1}+c+z&&\sum_1^{j}\mu_i+a_{j}+c&&\\
&&&\sum_1^{j}\mu_i+c&&&\\
&&\sum_1^{j}\mu_i+c&& \sum_1^{j+1}\mu_i+c+1&&
\end{array}
\end{align*}
Computing we have
\begin{align*}
&\max\left(\sum^{j-1}\mu_i+a_{j-1}+\sum^{j+1}\mu_i+1+2c,2\left(\sum^j\mu_i\right)+a_{j}+2c\right)-\sum^j\mu_i-c\\
=&\max\left(\sum^{j-1}\mu_i+\mu_{j+1}+a_{j-1}+c+1+z,\sum^{j-1}\mu_i+\mu_j+a_{j}+c\right)\\
\end{align*}
If $z=-1$, then this evaluates to $\sum^j \mu_i+a_{j}+c$. If $z=0$, then we have the same two expressions in the maximum as before, so two cases can occur once again, depending on whether $\mu_{j+1}+a_{j-1}<\mu_j+a_{j}$ or $\mu_{j+1}+a_{j-1}=\mu_j+a_{j}$. This corresponds to the current strip having differences $\omega_{k_j}^{(m-1)}$ or $\omega_{k_j+1}^{(m-1)}$. The same analysis applies to the next lattice point on the slanted portion of the break path and so forth. Hence, the special case at an out-elbow can propagate down a slanted portion of the break path as long as $\mu_{j+1}+a_{j-1}-a_{j}=\mu_{j}$ holds for the current strip $j$. Hence, after excavation each weight label along a revealed strip agrees with the weight label $\omega_{k_j}$, except for a collection of subsequences of indices $x,\ldots, y$ for which the weight label is $\omega_{k_j+1}$ precisely when $\mu_{j+1}+a_{j-1}-a_{j}=\mu_j$ for all $j\in \{x,\ldots,y\}$ and there is an out elbow along strip $x$. The possible break path changes are shown below.

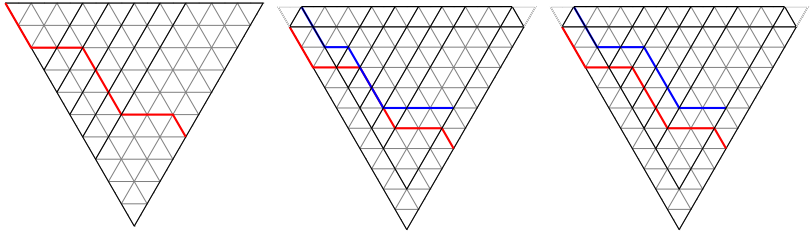
\begin{center}
\scalemath{.5}{
\begin{tikzpicture}
  \begin{ternaryaxis}[
  clip=false,
    xmin=0,
    ymin=0,
    zmin=0,
    xmax=10,
    ymax=10,
    zmax=10,
    yscale=-1,
    xtick={0, ..., 10},
    ytick={0, ..., 10},
    ztick={0, ..., 10},
    xticklabels={},
    yticklabels={},
    zticklabels={},
    major tick length=0,
        minor tick length=0,
            every axis grid/.style={gray},
    minor tick num=1,
    ]
          
   \addplot3[color=red,ultra thick] coordinates {(0,10,0)(2,8,0)(2,6,2)(5,3,2)(5,1,4)(6,1,4)};
   
   \addplot3[thick] coordinates {(0,7,3)(3,7,0)};
   \addplot3[thick] coordinates {(0,6,4)(4,6,0)};
   \addplot3[thick] coordinates {(0,5,5)(5,5,0)};
   \addplot3[thick] coordinates {(0,4,6)(6,4,0)};
   \addplot3[thick] coordinates {(0,3,7)(7,3,0)};
   
  \end{ternaryaxis}
\end{tikzpicture}
\begin{tikzpicture}
  \begin{ternaryaxis}[
  clip=false,
    xmin=0,
    ymin=0,
    zmin=0,
    xmax=11,
    ymax=11,
    zmax=11,
    yscale=-1,
    xtick={0, ..., 10},
    ytick={0, ..., 10},
    ztick={0, ..., 10},
    xticklabels={},
    yticklabels={},
    zticklabels={},
    major tick length=0,
        minor tick length=0,
            every axis grid/.style={gray},
    minor tick num=1,
    ]
          
   \addplot3[color=red,ultra thick] coordinates {(1,10,0)(3,8,0)(3,6,2)(6,3,2)(6,1,4)(7,1,4)};
   \addplot3[color=blue,ultra thick] coordinates {(0,10,1)(2,8,1)(2,7,2)(5,4,2)(5,1,5)};

   \addplot3[thick] coordinates {(1,0,10)(0,1,10)(0,10,1)(1,10,0)};
   \addplot3[thick] coordinates {(0,1,10)(9,1,1)(0,10,1)};
   \addplot3[thick] coordinates {(1,0,10)(1,10,0)};
   
   \addplot3[thick] coordinates {(0,7,4)(3,7,1)};
   \addplot3[thick] coordinates {(0,6,5)(4,6,1)};
   \addplot3[thick] coordinates {(0,5,6)(5,5,1)};
   \addplot3[thick] coordinates {(0,4,7)(6,4,1)};
   \addplot3[thick] coordinates {(0,3,8)(7,3,1)};
   
     \addplot3[color=white,thick] coordinates {(0,10,1)(0,11,0)(1,10,0)};
    \addplot3[color=white,thick] coordinates {(1,0,10)(0,0,11)(0,1,10)};
   
  \end{ternaryaxis}
\end{tikzpicture}
\begin{tikzpicture}
  \begin{ternaryaxis}[
  clip=false,
    xmin=0,
    ymin=0,
    zmin=0,
    xmax=11,
    ymax=11,
    zmax=11,
    yscale=-1,
    xtick={0, ..., 10},
    ytick={0, ..., 10},
    ztick={0, ..., 10},
    xticklabels={},
    yticklabels={},
    zticklabels={},
    major tick length=0,
        minor tick length=0,
            every axis grid/.style={gray},
    minor tick num=1,
    ]
          
   \addplot3[color=red,ultra thick] coordinates {(1,10,0)(3,8,0)(3,6,2)(6,3,2)(6,1,4)(7,1,4)};
   \addplot3[color=blue,ultra thick] coordinates {(0,10,1)(2,8,1)(2,6,3)(5,3,3)(5,1,5)};

   \addplot3[thick] coordinates {(1,0,10)(0,1,10)(0,10,1)(1,10,0)};
   \addplot3[thick] coordinates {(0,1,10)(9,1,1)(0,10,1)};
   \addplot3[thick] coordinates {(1,0,10)(1,10,0)};
   
   \addplot3[thick] coordinates {(0,7,4)(3,7,1)};
   \addplot3[thick] coordinates {(0,6,5)(4,6,1)};
   \addplot3[thick] coordinates {(0,5,6)(5,5,1)};
   \addplot3[thick] coordinates {(0,4,7)(6,4,1)};
   \addplot3[thick] coordinates {(0,3,8)(7,3,1)};
   
    \addplot3[color=white,thick] coordinates {(0,10,1)(0,11,0)(1,10,0)};
    \addplot3[color=white,thick] coordinates {(1,0,10)(0,0,11)(0,1,10)};
   
  \end{ternaryaxis}
\end{tikzpicture}}
\end{center}

Let $C$ be a connected component of the ribbon shape $\nu/\lambda$ and suppose that the cells of $C$ are contained in the rows numbered $x,\ldots, y$. Consider first the case that $C$ contains a $1$ entry in the top row $x$. Then $C$ contains a single $2$ in each of the rows $x+1,\ldots,y$. In particular, the break path has a slanted step between strips $j-1$ and $j$ for $j$ such that $x+1\leq j\leq y$ and a horizontal step between strips $x-1$ and $x$. In particular, there is an out elbow along strip $x$. Furthermore, $\mu_{j+1}=\mu_j-(a_{j-1}-a_{j})$ for all $j$ such that $x+1\leq j\leq y$ and $\mu_{y+1}< \mu_y-(a_{y-1}-a_{y})$. This implies that the above-mentioned special case occurs along each strip from strip $x$ to strip $y-1$. The new break path has a slanted step between strips $j-1$ and $j$ for all $j$ such that $x\leq j\leq y-1$ and horizontal step between strips $y-1$ and $y$. 

Now suppose that the connected component $C$ contains a $2$ in each of the rows $x,\ldots, y$. Then the break path has a slanted step between strips $j-1$ and $j$ for all $j$ in the range $x\leq j\leq y$. But $\mu_{x}<\mu_{x-1}-(a_{x-2}-a_{x-1})$, so the special case cannot occur along strip $x$, and therefore neither along any strip $j$ in the range $x\leq j\leq y$. Hence, this agrees with the mixed BK move.
\end{proof}

\begin{figure}[b]
\centering
\minipage[b]{0.5\textwidth}
\centering
\scalemath{.4}{
\begin{tikzpicture}
  \begin{ternaryaxis}[ylabel=$\scalemath{2}{\sigma_6}$,xlabel=$\scalemath{2}{\lambda}$,zlabel=$\scalemath{2}{\mu}$,
    clip=false,xticklabels={},yticklabels={},zticklabels={},xmin=0,
    ymin=0,
    zmin=0,
    xmax=14,
    ymax=14,
    zmax=14,
    xtick={2,4,6,8,10,12},
    ytick={2,4,6,8,10,12},
    ztick={2,4,6,8,10,12},
    major tick length=0,
]    
  	\node at (axis cs:0,24,0) {$\scalemath{1.75}{0}$};
	\node at (axis cs:2,22,0) {$\scalemath{1.75}{6}$};
	\node at (axis cs:4,20,0) {$\scalemath{1.75}{6}$};
	\node at (axis cs:6,18,0) {$\scalemath{1.75}{6}$};
	\node at (axis cs:8,16,0) {$\scalemath{1.75}{6}$};
	\node at (axis cs:10,14,0) {$\scalemath{1.75}{6}$};
	\node at (axis cs:12,12,0) {$\scalemath{1.75}{6}$};
	\node at (axis cs:14,10,0) {$\scalemath{1.75}{6}$};
	
	\node at (axis cs:0,22,2) {$\scalemath{1.75}{9}$};
	\node at (axis cs:2,20,2) {$\scalemath{1.75}{14}$};
	\node at (axis cs:4,18,2) {$\scalemath{1.75}{14}$};
	\node at (axis cs:6,16,2) {$\scalemath{1.75}{14}$};
	\node at (axis cs:8,14,2) {$\scalemath{1.75}{14}$};
	\node at (axis cs:10,12,2) {$\scalemath{1.75}{14}$};
	\node at (axis cs:12,10,2) {$\scalemath{1.75}{14}$};
	
	\node at (axis cs:0,20,4) {$\scalemath{1.75}{15}$};
	\node at (axis cs:2,18,4) {$\scalemath{1.75}{20}$};
	\node at (axis cs:4,16,4) {$\scalemath{1.75}{20}$};
	\node at (axis cs:6,14,4) {$\scalemath{1.75}{20}$};
	\node at (axis cs:8,12,4) {$\scalemath{1.75}{20}$};
	\node at (axis cs:10,10,4) {$\scalemath{1.75}{20}$};
	
	\node at (axis cs:0,18,6) {$\scalemath{1.75}{21}$};
	\node at (axis cs:2,16,6) {$\scalemath{1.75}{23}$};
	\node at (axis cs:4,14,6) {$\scalemath{1.75}{23}$};
	\node at (axis cs:6,12,6) {$\scalemath{1.75}{23}$};
	\node at (axis cs:8,10,6) {$\scalemath{1.75}{23}$};
	
	\node at (axis cs:0,16,8) {$\scalemath{1.75}{24}$};
	\node at (axis cs:2,14,8) {$\scalemath{1.75}{26}$};
	\node at (axis cs:4,12,8) {$\scalemath{1.75}{26}$};
	\node at (axis cs:6,10,8) {$\scalemath{1.75}{26}$};
	
	\node at (axis cs:0,14,10) {$\scalemath{1.75}{27}$};
	\node at (axis cs:2,12,10) {$\scalemath{1.75}{27}$};
	\node at (axis cs:4,10,10) {$\scalemath{1.75}{27}$};
	
	\node at (axis cs:0,12,12) {$\scalemath{1.75}{28}$};
	\node at (axis cs:2,10,12) {$\scalemath{1.75}{28}$};
	
	\node at (axis cs:0,10,14) {$\scalemath{1.75}{29}$};

  \end{ternaryaxis}
\end{tikzpicture}}

\vspace{-5.2mm}

\hspace{.0000000001mm}

\scalemath{.4}{
\begin{tikzpicture}
  \begin{ternaryaxis}[ylabel=$\scalemath{2}{\nu}$,xlabel=$\scalemath{2}{\omega_4}$,xlabel style={xshift=0.5cm},ylabel style={xshift=-0.5cm},
    clip=false,xticklabels={},yticklabels={},zticklabels={},xmin=0,
    ymin=0,
    zmin=0,
    xmax=14,
    ymax=14,
    zmax=14,
    yscale=-1,
    xtick={2,4,6,8,10,12},
    ytick={2,4,6,8,10,12},
    ztick={2,4,6,8,10,12},
    major tick length=0,
]    
  	\node at (axis cs:0,24,0) {$\scalemath{1.75}{0}$};
	
	\node at (axis cs:0,22,2) {$\scalemath{1.75}{9}$};
	\node at (axis cs:2,22,0) {$\scalemath{1.75}{10}$};
	
	\node at (axis cs:0,20,4) {$\scalemath{1.75}{15}$};
	\node at (axis cs:2,20,2) {$\scalemath{1.75}{16}$};
	\node at (axis cs:4,20,0) {$\scalemath{1.75}{16}$};
	
	\node at (axis cs:0,18,6) {$\scalemath{1.75}{21}$};
	\node at (axis cs:2,18,4) {$\scalemath{1.75}{22}$};
	\node at (axis cs:4,18,2) {$\scalemath{1.75}{22}$};
	\node at (axis cs:6,18,0) {$\scalemath{1.75}{22}$};
	
	\node at (axis cs:0,16,8) {$\scalemath{1.75}{24}$};
	\node at (axis cs:2,16,6) {$\scalemath{1.75}{25}$};
	\node at (axis cs:4,16,4) {$\scalemath{1.75}{26}$};
	\node at (axis cs:6,16,2) {$\scalemath{1.75}{26}$};
	\node at (axis cs:8,16,0) {$\scalemath{1.75}{26}$};
	
	\node at (axis cs:0,14,10) {$\scalemath{1.75}{27}$};
	\node at (axis cs:2,14,8) {$\scalemath{1.75}{28}$};
	\node at (axis cs:4,14,6) {$\scalemath{1.75}{29}$};
	\node at (axis cs:6,14,4) {$\scalemath{1.75}{30}$};
	\node at (axis cs:8,14,2) {$\scalemath{1.75}{30}$};
	\node at (axis cs:10,14,0) {$\scalemath{1.75}{30}$};
	
	\node at (axis cs:0,12,12) {$\scalemath{1.75}{28}$};
	\node at (axis cs:2,12,10) {$\scalemath{1.75}{29}$};
	\node at (axis cs:4,12,8) {$\scalemath{1.75}{30}$};
	\node at (axis cs:6,12,6) {$\scalemath{1.75}{31}$};
	\node at (axis cs:8,12,4) {$\scalemath{1.75}{32}$};
	\node at (axis cs:10,12,2) {$\scalemath{1.75}{32}$};
	\node at (axis cs:12,12,0) {$\scalemath{1.75}{32}$};
	
	\node at (axis cs:0,10,14) {$\scalemath{1.75}{29}$};
	\node at (axis cs:2,10,12) {$\scalemath{1.75}{30}$};
	\node at (axis cs:4,10,10) {$\scalemath{1.75}{31}$};
	\node at (axis cs:6,10,8) {$\scalemath{1.75}{32}$};
	\node at (axis cs:8,10,6) {$\scalemath{1.75}{33}$};
	\node at (axis cs:10,10,4) {$\scalemath{1.75}{33}$};
	\node at (axis cs:12,10,2) {$\scalemath{1.75}{33}$};
	\node at (axis cs:14,10,0) {$\scalemath{1.75}{33}$};
	
	\addplot3[color=red,ultra thick] coordinates {(0,14,0)(2,12,0)(2,8,4)(8,2,4)(8,0,6)};

  \end{ternaryaxis}
\end{tikzpicture}}
\subcaption{top faces \label{top break paths}}
\endminipage
\minipage[b]{0.5\textwidth}
\centering
\scalemath{.4}{
\begin{tikzpicture}
  \begin{ternaryaxis}[rotate=-30,ylabel=$\scalemath{2}{\sigma_6}$,xlabel=$\scalemath{2}{\rho}$,zlabel=$\scalemath{2}{\nu}$,zlabel style={xshift=-0.5cm},clip=false,xticklabels={},yticklabels={},zticklabels={},xmin=0,
    ymin=0,
    zmin=0,
    xmax=14,
    ymax=14,
    zmax=14,
    xtick={2,4,6,8,10,12},
    ytick={2,4,6,8,10,12},
    ztick={2,4,6,8,10,12},
    major tick length=0,
]    
  	\node at (axis cs:0,24,0) {$\scalemath{1.75}{0}$};
	\node at (axis cs:2,22,0) {$\scalemath{1.75}{6}$};
	\node at (axis cs:4,20,0) {$\scalemath{1.75}{6}$};
	\node at (axis cs:6,18,0) {$\scalemath{1.75}{6}$};
	\node at (axis cs:8,16,0) {$\scalemath{1.75}{6}$};
	\node at (axis cs:10,14,0) {$\scalemath{1.75}{6}$};
	\node at (axis cs:12,12,0) {$\scalemath{1.75}{6}$};
	\node at (axis cs:14,10,0) {$\scalemath{1.75}{6}$};
	
	\node at (axis cs:0,22,2) {$\scalemath{1.75}{10}$};
	\node at (axis cs:2,20,2) {$\scalemath{1.75}{15}$};
	\node at (axis cs:4,18,2) {$\scalemath{1.75}{15}$};
	\node at (axis cs:6,16,2) {$\scalemath{1.75}{15}$};
	\node at (axis cs:8,14,2) {$\scalemath{1.75}{15}$};
	\node at (axis cs:10,12,2) {$\scalemath{1.75}{15}$};
	\node at (axis cs:12,10,2) {$\scalemath{1.75}{15}$};
	
	\node at (axis cs:0,20,4) {$\scalemath{1.75}{16}$};
	\node at (axis cs:2,18,4) {$\scalemath{1.75}{21}$};
	\node at (axis cs:4,16,4) {$\scalemath{1.75}{21}$};
	\node at (axis cs:6,14,4) {$\scalemath{1.75}{21}$};
	\node at (axis cs:8,12,4) {$\scalemath{1.75}{21}$};
	\node at (axis cs:10,10,4) {$\scalemath{1.75}{21}$};
	
	\node at (axis cs:0,18,6) {$\scalemath{1.75}{22}$};
	\node at (axis cs:2,16,6) {$\scalemath{1.75}{25}$};
	\node at (axis cs:4,14,6) {$\scalemath{1.75}{25}$};
	\node at (axis cs:6,12,6) {$\scalemath{1.75}{25}$};
	\node at (axis cs:8,10,6) {$\scalemath{1.75}{25}$};
	
	\node at (axis cs:0,16,8) {$\scalemath{1.75}{26}$};
	\node at (axis cs:2,14,8) {$\scalemath{1.75}{29}$};
	\node at (axis cs:4,12,8) {$\scalemath{1.75}{29}$};
	\node at (axis cs:6,10,8) {$\scalemath{1.75}{29}$};
	
	\node at (axis cs:0,14,10) {$\scalemath{1.75}{30}$};
	\node at (axis cs:2,12,10) {$\scalemath{1.75}{31}$};
	\node at (axis cs:4,10,10) {$\scalemath{1.75}{31}$};
	
	\node at (axis cs:0,12,12) {$\scalemath{1.75}{32}$};
	\node at (axis cs:2,10,12) {$\scalemath{1.75}{32}$};
	
	\node at (axis cs:0,10,14) {$\scalemath{1.75}{33}$};

  \end{ternaryaxis}
\end{tikzpicture}}
\scalemath{.4}{
\begin{tikzpicture}
  \begin{ternaryaxis}[rotate=-30,zlabel=$\scalemath{2}{\lambda}$, zlabel style={yshift=.5cm,xshift=1cm}, xlabel=$\scalemath{2}{\omega_4}$,xlabel style={yshift=-.5cm, xshift=0.5cm},
    clip=false,xticklabels={},yticklabels={},zticklabels={},xmin=0,
    ymin=0,
    zmin=0,
    xmax=14,
    ymax=14,
    zmax=14,
    yscale=-1,
    xtick={2,4,6,8,10,12},
    ytick={2,4,6,8,10,12},
    ztick={2,4,6,8,10,12},
    major tick length=0,
]    
  	\node at (axis cs:0,24,0) {$\scalemath{1.75}{6}$};
	
	\node at (axis cs:0,22,2) {$\scalemath{1.75}{14}$};
	\node at (axis cs:2,22,0) {$\scalemath{1.75}{15}$};
	
	\node at (axis cs:0,20,4) {$\scalemath{1.75}{20}$};
	\node at (axis cs:2,20,2) {$\scalemath{1.75}{21}$};
	\node at (axis cs:4,20,0) {$\scalemath{1.75}{21}$};
	
	\node at (axis cs:0,18,6) {$\scalemath{1.75}{23}$};
	\node at (axis cs:2,18,4) {$\scalemath{1.75}{24}$};
	\node at (axis cs:4,18,2) {$\scalemath{1.75}{25}$};
	\node at (axis cs:6,18,0) {$\scalemath{1.75}{25}$};
	
	\node at (axis cs:0,16,8) {$\scalemath{1.75}{26}$};
	\node at (axis cs:2,16,6) {$\scalemath{1.75}{27}$};
	\node at (axis cs:4,16,4) {$\scalemath{1.75}{28}$};
	\node at (axis cs:6,16,2) {$\scalemath{1.75}{29}$};
	\node at (axis cs:8,16,0) {$\scalemath{1.75}{29}$};
	
	\node at (axis cs:0,14,10) {$\scalemath{1.75}{27}$};
	\node at (axis cs:2,14,8) {$\scalemath{1.75}{28}$};
	\node at (axis cs:4,14,6) {$\scalemath{1.75}{29}$};
	\node at (axis cs:6,14,4) {$\scalemath{1.75}{30}$};
	\node at (axis cs:8,14,2) {$\scalemath{1.75}{31}$};
	\node at (axis cs:10,14,0) {$\scalemath{1.75}{31}$};
	
	\node at (axis cs:0,12,12) {$\scalemath{1.75}{28}$};
	\node at (axis cs:2,12,10) {$\scalemath{1.75}{29}$};
	\node at (axis cs:4,12,8) {$\scalemath{1.75}{30}$};
	\node at (axis cs:6,12,6) {$\scalemath{1.75}{31}$};
	\node at (axis cs:8,12,4) {$\scalemath{1.75}{32}$};
	\node at (axis cs:10,12,2) {$\scalemath{1.75}{32}$};
	\node at (axis cs:12,12,0) {$\scalemath{1.75}{32}$};
	
	\node at (axis cs:0,10,14) {$\scalemath{1.75}{29}$};
	\node at (axis cs:2,10,12) {$\scalemath{1.75}{30}$};
	\node at (axis cs:4,10,10) {$\scalemath{1.75}{31}$};
	\node at (axis cs:6,10,8) {$\scalemath{1.75}{32}$};
	\node at (axis cs:8,10,6) {$\scalemath{1.75}{33}$};
	\node at (axis cs:10,10,4) {$\scalemath{1.75}{33}$};
	\node at (axis cs:12,10,2) {$\scalemath{1.75}{33}$};
	\node at (axis cs:14,10,0) {$\scalemath{1.75}{33}$};
	
	\addplot3[color=red,ultra thick] coordinates {(0,14,0)(2,12,0)(2,10,2)(8,4,2)(8,0,6)};

  \end{ternaryaxis}
\end{tikzpicture}}
\subcaption{bottom faces\label{bottom break paths}}
\endminipage
\caption{The top and bottom of a $4$-hive with break paths indicated. \label{fig:4hive example}}
\end{figure}

\begin{figure}
\centering
\minipage[b]{0.5\textwidth}
\centering
\scalemath{.4}{
\begin{tikzpicture}
  \begin{ternaryaxis}[ylabel=$\scalemath{2}{\sigma_6}$,xlabel=$\scalemath{2}{\lambda}$,
    clip=false,xticklabels={},yticklabels={},zticklabels={},xmin=0,
    ymin=0,
    zmin=0,
    xmax=14,
    ymax=14,
    zmax=14,
    xtick={2,4,6,8,10,12},
    ytick={2,4,6,8,10,12},
    ztick={2,4,6,8,10,12},
    major tick length=0,
]    
  	\node at (axis cs:0,24,0) {$\scalemath{1.75}{0}$};
	\node at (axis cs:2,22,0) {$\scalemath{1.75}{6}$};
	\node at (axis cs:4,20,0) {$\scalemath{1.75}{6}$};
	\node at (axis cs:6,18,0) {$\scalemath{1.75}{6}$};
	\node at (axis cs:8,16,0) {$\scalemath{1.75}{6}$};
	\node at (axis cs:10,14,0) {$\scalemath{1.75}{6}$};
	\node at (axis cs:12,12,0) {$\scalemath{1.75}{6}$};
	\node at (axis cs:14,10,0) {$\scalemath{1.75}{6}$};
	
	\node at (axis cs:0,22,2) {$\scalemath{1.75}{9}$};
	\node at (axis cs:2,20,2) {$\scalemath{1.75}{14}$};
	\node at (axis cs:4,18,2) {$\scalemath{1.75}{14}$};
	\node at (axis cs:6,16,2) {$\scalemath{1.75}{14}$};
	\node at (axis cs:8,14,2) {$\scalemath{1.75}{14}$};
	\node at (axis cs:10,12,2) {$\scalemath{1.75}{14}$};
	\node at (axis cs:12,10,2) {$\scalemath{1.75}{14}$};
	
	\node at (axis cs:0,20,4) {$\scalemath{1.75}{15}$};
	\node at (axis cs:2,18,4) {$\scalemath{1.75}{20}$};
	\node at (axis cs:4,16,4) {$\scalemath{1.75}{20}$};
	\node at (axis cs:6,14,4) {$\scalemath{1.75}{20}$};
	\node at (axis cs:8,12,4) {$\scalemath{1.75}{20}$};
	\node at (axis cs:10,10,4) {$\scalemath{1.75}{20}$};
	
	\node at (axis cs:0,18,6) {$\scalemath{1.75}{21}$};
	\node at (axis cs:2,16,6) {$\scalemath{1.75}{23}$};
	\node at (axis cs:4,14,6) {$\scalemath{1.75}{23}$};
	\node at (axis cs:6,12,6) {$\scalemath{1.75}{23}$};
	\node at (axis cs:8,10,6) {$\scalemath{1.75}{23}$};
	
	\node at (axis cs:0,16,8) {$\scalemath{1.75}{24}$};
	\node at (axis cs:2,14,8) {$\scalemath{1.75}{26}$};
	\node at (axis cs:4,12,8) {$\scalemath{1.75}{26}$};
	\node at (axis cs:6,10,8) {$\scalemath{1.75}{26}$};
	
	\node at (axis cs:0,14,10) {$\scalemath{1.75}{27}$};
	\node at (axis cs:2,12,10) {$\scalemath{1.75}{27}$};
	\node at (axis cs:4,10,10) {$\scalemath{1.75}{27}$};
	
	\node at (axis cs:0,12,12) {$\scalemath{1.75}{28}$};
	\node at (axis cs:2,10,12) {$\scalemath{1.75}{28}$};
	
	\node at (axis cs:0,10,14) {$\scalemath{1.75}{29}$};

  \end{ternaryaxis}
\end{tikzpicture}}

\vspace{-3.1mm}

\scalemath{.4}{
\begin{tikzpicture}
  \begin{ternaryaxis}[ylabel=$\scalemath{2}{\nu}$,xlabel=$\scalemath{2}{\omega_4}$,xlabel style={xshift=0.5cm},ylabel style={xshift=-0.5cm},
    clip=false,xticklabels={},yticklabels={},zticklabels={},xmin=0,
    ymin=0,
    zmin=0,
    xmax=14,
    ymax=14,
    zmax=14,
    yscale=-1,
    xtick={2,4,6,8,10,12},
    ytick={2,4,6,8,10,12},
    ztick={2,4,6,8,10,12},
    major tick length=0,
]    
  	\node at (axis cs:0,24,0) {$\scalemath{1.75}{0}$};
	
	\node at (axis cs:0,22,2) {$\scalemath{1.75}{9}$};
	\node at (axis cs:2,22,0) {$\scalemath{1.75}{10}$};
	
	\node at (axis cs:0,20,4) {$\scalemath{1.75}{15}$};
	\node at (axis cs:2,20,2) {$\scalemath{1.75}{16}$};
	\node at (axis cs:4,20,0) {$\scalemath{1.75}{16}$};
	
	\node at (axis cs:0,18,6) {$\scalemath{1.75}{21}$};
	\node at (axis cs:2,18,4) {$\scalemath{1.75}{22}$};
	\node at (axis cs:4,18,2) {$\scalemath{1.75}{22}$};
	\node at (axis cs:6,18,0) {$\scalemath{1.75}{22}$};
	
	\node at (axis cs:0,16,8) {$\scalemath{1.75}{24}$};
	\node at (axis cs:2,16,6) {$\scalemath{1.75}{25}$};
	\node at (axis cs:4,16,4) {$\scalemath{1.75}{26}$};
	\node at (axis cs:6,16,2) {$\scalemath{1.75}{26}$};
	\node at (axis cs:8,16,0) {$\scalemath{1.75}{26}$};
	
	\node at (axis cs:0,14,10) {$\scalemath{1.75}{27}$};
	\node at (axis cs:2,14,8) {$\scalemath{1.75}{28}$};
	\node at (axis cs:4,14,6) {$\scalemath{1.75}{29}$};
	\node at (axis cs:6,14,4) {$\scalemath{1.75}{30}$};
	\node at (axis cs:8,14,2) {$\scalemath{1.75}{30}$};
	\node at (axis cs:10,14,0) {$\scalemath{1.75}{30}$};
	
	\node at (axis cs:0,12,12) {$\scalemath{1.75}{28}$};
	\node at (axis cs:2,12,10) {$\scalemath{1.75}{29}$};
	\node at (axis cs:4,12,8) {$\scalemath{1.75}{30}$};
	\node at (axis cs:6,12,6) {$\scalemath{1.75}{31}$};
	\node at (axis cs:8,12,4) {$\scalemath{1.75}{32}$};
	\node at (axis cs:10,12,2) {$\scalemath{1.75}{32}$};
	\node at (axis cs:12,12,0) {$\scalemath{1.75}{32}$};
	
	\node at (axis cs:0,10,14) {$\scalemath{1.75}{29}$};
	\node at (axis cs:2,10,12) {$\scalemath{1.75}{30}$};
	\node at (axis cs:4,10,10) {$\scalemath{1.75}{31}$};
	\node at (axis cs:6,10,8) {$\scalemath{1.75}{32}$};
	\node at (axis cs:8,10,6) {$\scalemath{1.75}{33}$};
	\node at (axis cs:10,10,4) {$\scalemath{1.75}{33}$};
	\node at (axis cs:12,10,2) {$\scalemath{1.75}{33}$};
	\node at (axis cs:14,10,0) {$\scalemath{1.75}{33}$};
	
	\addplot3[color=red,ultra thick] coordinates {(0,14,0)(2,12,0)(2,8,4)(8,2,4)(8,0,6)};

  \end{ternaryaxis}
\end{tikzpicture}}
\endminipage
\minipage[b]{0.5\textwidth}
\centering
\scalemath{.4}{
\begin{tikzpicture}
  \begin{ternaryaxis}[ylabel=$\scalemath{2}{\sigma_6}$,xlabel=$\scalemath{2}{\lambda}$,
    clip=false,xticklabels={},yticklabels={},zticklabels={},xmin=0,
    z axis line style= { draw opacity=.1 },
    ymin=0,
    zmin=0,
    xmax=14,
    ymax=14,
    zmax=14,
    xtick={2,4,6,8,10,12},
    ytick={2,4,6,8,10,12},
    ztick={2,4,6,8,10,12},
    major tick length=0,
]    
  	\node at (axis cs:0,24,0) {$\scalemath{1.75}{0}$};
	\node at (axis cs:2,22,0) {$\scalemath{1.75}{6}$};
	\node at (axis cs:4,20,0) {$\scalemath{1.75}{6}$};
	\node at (axis cs:6,18,0) {$\scalemath{1.75}{6}$};
	\node at (axis cs:8,16,0) {$\scalemath{1.75}{6}$};
	\node at (axis cs:10,14,0) {$\scalemath{1.75}{6}$};
	\node at (axis cs:12,12,0) {$\scalemath{1.75}{6}$};
	\node at (axis cs:14,10,0) {$\scalemath{1.75}{6}$};
	
	\node at (axis cs:0,22,2) {$\scalemath{1.75}{15}$};
	\node at (axis cs:2,20,2) {$\scalemath{1.75}{14}$};
	\node at (axis cs:4,18,2) {$\scalemath{1.75}{14}$};
	\node at (axis cs:6,16,2) {$\scalemath{1.75}{14}$};
	\node at (axis cs:8,14,2) {$\scalemath{1.75}{14}$};
	\node at (axis cs:10,12,2) {$\scalemath{1.75}{14}$};
	\node at (axis cs:12,10,2) {$\scalemath{1.75}{14}$};
	
	\node at (axis cs:0,20,4) {$\scalemath{1.75}{21}$};
	\node at (axis cs:2,18,4) {$\scalemath{1.75}{20}$};
	\node at (axis cs:4,16,4) {$\scalemath{1.75}{20}$};
	\node at (axis cs:6,14,4) {$\scalemath{1.75}{20}$};
	\node at (axis cs:8,12,4) {$\scalemath{1.75}{20}$};
	\node at (axis cs:10,10,4) {$\scalemath{1.75}{20}$};
	
	\node at (axis cs:0,18,6) {$\scalemath{1.75}{24}$};
	\node at (axis cs:2,16,6) {$\scalemath{1.75}{23}$};
	\node at (axis cs:4,14,6) {$\scalemath{1.75}{23}$};
	\node at (axis cs:6,12,6) {$\scalemath{1.75}{23}$};
	\node at (axis cs:8,10,6) {$\scalemath{1.75}{23}$};
	
	\node at (axis cs:0,16,8) {$\scalemath{1.75}{27}$};
	\node at (axis cs:2,14,8) {$\scalemath{1.75}{26}$};
	\node at (axis cs:4,12,8) {$\scalemath{1.75}{26}$};
	\node at (axis cs:6,10,8) {$\scalemath{1.75}{26}$};
	
	\node at (axis cs:0,14,10) {$\scalemath{1.75}{28}$};
	\node at (axis cs:2,12,10) {$\scalemath{1.75}{27}$};
	\node at (axis cs:4,10,10) {$\scalemath{1.75}{27}$};
	
	\node at (axis cs:0,12,12) {$\scalemath{1.75}{29}$};
	\node at (axis cs:2,10,12) {$\scalemath{1.75}{28}$};
	
	\node at (axis cs:0,10,14) {$\scalemath{1.75}{29}$};
	
	\fill[gray, opacity=.5] (axis cs:0,22,2)--(axis cs:0,24,0)--(axis cs:2,22,0)--cycle;
         \fill[gray, opacity=.5] (axis cs:0,0,14)--(axis cs:2,0,12)--(axis cs:0,2,12)--cycle;
         
         \addplot3[color=black] coordinates {(2,12,0)(2,0,12)};
	\addplot3[color=black, thick] coordinates {(0,12,2)(2,12,0)};
	\addplot3[color=black, thick] coordinates {(0,2,12)(2,0,12)};
	
	\addplot3[color=red,ultra thick] coordinates {(0,12,2)(2,12,0)};

  \end{ternaryaxis}
\end{tikzpicture}}

\vspace{-3.1mm}

\scalemath{.4}{
\begin{tikzpicture}
  \begin{ternaryaxis}[
    clip=false,xticklabels={},yticklabels={},zticklabels={},xmin=0,
    z axis line style= { draw opacity=.1 },
    ymin=0,
    zmin=0,
    xmax=14,
    ymax=14,
    zmax=14,
    yscale=-1,
    xtick={2,4,6,8,10,12},
    ytick={2,4,6,8,10,12},
    ztick={2,4,6,8,10,12},
    major tick length=0,
]    
	\fill[gray,opacity=.5] (axis cs:0,14,0)--(axis cs:14,0,0)--(axis cs:10,2,2)--(axis cs:0,12,2)--cycle;
	\fill[gray,opacity=.5] (axis cs:0,0,14)--(axis cs:14,0,0)--(axis cs:10,2,2)--(axis cs:0,2,12)--cycle;
	
  	\node[color=gray] at (axis cs:0,24,0) {$\scalemath{1.75}{0}$};
	
	\node at (axis cs:0,22,2) {$\scalemath{1.75}{15}$};
	\node[color=gray] at (axis cs:2,22,0) {$\scalemath{1.75}{10}$};
	
	\node at (axis cs:0,20,4) {$\scalemath{1.75}{21}$};
	\node at (axis cs:2,20,2) {$\scalemath{1.75}{21}$};
	\node[color=gray] at (axis cs:4,20,0) {$\scalemath{1.75}{16}$};
	
	\node at (axis cs:0,18,6) {$\scalemath{1.75}{24}$};
	\node at (axis cs:2,18,4) {$\scalemath{1.75}{25}$};
	\node at (axis cs:4,18,2) {$\scalemath{1.75}{25}$};
	\node[color=gray] at (axis cs:6,18,0) {$\scalemath{1.75}{22}$};
	
	\node at (axis cs:0,16,8) {$\scalemath{1.75}{27}$};
	\node at (axis cs:2,16,6) {$\scalemath{1.75}{28}$};
	\node at (axis cs:4,16,4) {$\scalemath{1.75}{29}$};
	\node at (axis cs:6,16,2) {$\scalemath{1.75}{29}$};
	\node[color=gray] at (axis cs:8,16,0) {$\scalemath{1.75}{26}$};
	
	\node at (axis cs:0,14,10) {$\scalemath{1.75}{28}$};
	\node at (axis cs:2,14,8) {$\scalemath{1.75}{29}$};
	\node at (axis cs:4,14,6) {$\scalemath{1.75}{30}$};
	\node at (axis cs:6,14,4) {$\scalemath{1.75}{31}$};
	\node at (axis cs:8,14,2) {$\scalemath{1.75}{31}$};
	\node[color=gray] at (axis cs:10,14,0) {$\scalemath{1.75}{30}$};
	
	\node at (axis cs:0,12,12) {$\scalemath{1.75}{29}$};
	\node at (axis cs:2,12,10) {$\scalemath{1.75}{30}$};
	\node at (axis cs:4,12,8) {$\scalemath{1.75}{31}$};
	\node at (axis cs:6,12,6) {$\scalemath{1.75}{32}$};
	\node at (axis cs:8,12,4) {$\scalemath{1.75}{32}$};
	\node at (axis cs:10,12,2) {$\scalemath{1.75}{32}$};
	\node[color=gray] at (axis cs:12,12,0) {$\scalemath{1.75}{32}$};
	
	\node[color=gray] at (axis cs:0,10,14) {$\scalemath{1.75}{29}$};
	\node[color=gray] at (axis cs:2,10,12) {$\scalemath{1.75}{30}$};
	\node[color=gray] at (axis cs:4,10,10) {$\scalemath{1.75}{31}$};
	\node[color=gray] at (axis cs:6,10,8) {$\scalemath{1.75}{32}$};
	\node[color=gray] at (axis cs:8,10,6) {$\scalemath{1.75}{33}$};
	\node[color=gray] at (axis cs:10,10,4) {$\scalemath{1.75}{33}$};
	\node[color=gray] at (axis cs:12,10,2) {$\scalemath{1.75}{33}$};
	\node[color=gray] at (axis cs:14,10,0) {$\scalemath{1.75}{33}$};
	
	\addplot3[color=red,ultra thick] coordinates {(0,12,2)(0,10,4)(6,4,4)(6,2,6)};
	
	\addplot3[color=black, thick] coordinates {(0,2,12)(10,2,2)};
	\addplot3[color=black, thick] coordinates {(0,12,2)(10,2,2)};

  \end{ternaryaxis}
\end{tikzpicture}}
\endminipage
\caption{The revealed hive values and break path after excavation.}
\end{figure}

We close this section by defining $SL_m$ hives (see \cite{GS, Le1} for this definition). Recall that the $SL_m$ weight lattice is $\mathbb Z^{m-1}\cong\mathbb Z^m/(1,\ldots,1)$. Given an $SL_m$ dominant weight $\lambda$ represented by $(\lambda_1,\ldots,\lambda_m)$ define $\hat \lambda=(\lambda_1,\ldots,\lambda_m)-\frac{\sum \lambda_i}{m}(1,\ldots,1)$. Then $\hat \lambda\in \frac{1}{m}\mathbb Z^m$ doesn't depend on the representative for $\lambda$ and $\lvert \hat\lambda\rvert=0$.

An $SL_m$ $n$-hive is an assignment $f_{i_1,\ldots,i_n}\in \frac{1}{m}\mathbb Z$ to each point of $\Delta^n_m$ such that the rhombus inequalities are satisfied on every two-dimensional face and the octahedron recurrence holds. An $SL_m$ $n$-hive has type $(\lambda^1,\ldots,\lambda^n)$ if for all $i$ the differences of the hive values along the edge from $me_i$ to $me_{i+1}$ give $\hat \lambda^i$. We have the analogue of Theorem \ref{thm:hives} in the $SL_m$ setting.

\section{Cyclic Sieving on Rectangular $\delta$-semistandard Tableaux}\label{sec:CSP}
\subsection{Rotation of the Satake Basis}
In this section, we recall the geometric Satake equivalence, describe the Satake basis for $\Inv(\vec\lambda)$, show that it is indexed by $\delta$-semistandard tableaux, and prove that rotation of tensor factors coincides (up to sign) with $\delta$-promotion of the tableaux. We accomplish the last goal by embedding $\Inv(\vec\lambda)$ in $\Inv(\vec\mu)$ where $\vec\mu$ is a minuscule sequence, and applying the corresponding result for minuscule sequences from \cite{FK}. 

\subsubsection{The Geometric Satake Equivalence}
The geometric Satake equivalence of Lusztig \cite{Lus}, Ginzburg \cite{Gin}, Beilinson--Drinfeld \cite{BD}, Mirkovi\'c--Vilonen \cite{MV} is an equivalence between the tensor category of representations of $G$ and the tensor category of perverse sheaves on the affine Grassmannian $Gr_{G^L}$ for the Langlands dual group $G^L$. In the present case, $G=SL_m$ and $G^L=PGL_m$, so let $Gr$ denote $Gr_{PGL_m}$. Let $\mathcal K=\mathbb C((t))$ and $\mathcal O=\mathbb C[[t]]$. As a set the affine Grassmannian is the quotient $Gr=PGL_m(\mathcal K)/PGL_m(\mathcal O)$, which can be thought of as the set of $\mathcal O$-lattices up to homothety. The affine Grassmannian is an ind-variety, an inductive limit of finite dimensional varieties. 

Let $\lambda$ be a dominant weight of $SL_m$ (equivalently a dominant coweight of $PGL_m$) and let $t^\lambda\in Gr$ be the class of the matrix with entries $t^{\lambda_i}$ along the diagonal. Then $Gr$ is stratified by the $PGL_m(\mathcal O)$-orbits, which are indexed by dominant $SL_m$-weights and are equal to $Gr(\lambda)=PGL_m(\mathcal O)\cdot t^{\lambda}$. The closure is $\overline{Gr}(\vec\lambda)=\cup_{\mu\leq \lambda} Gr(\mu)$ where weights are ordered by the usual dominance order, that is $\mu\leq \lambda$ if $\lambda-\mu$ is a nonnegative integer sum of simple roots. This stratification defines a dominant-weight-valued distance function between points of $Gr$. Given two points $p,q\in Gr$ there is a $g$ in $PGL_m(\mathcal K)$ such that $(g\cdot p,g\cdot q)=(t^0,t^\lambda)$ for a unique dominant weight $\lambda$. Then define $d(p,q)=\lambda$. Put differently, the $H$-orbits on $K/H$ are in bijection with $K$-orbits on $K/H\times K/H$ for any group $K$ and subgroup $H$. 

Let $\vec\lambda$ be an arbitrary sequence of dominant weights. Under the geometric Satake equivalence, the invariant space $\Inv(\vec{\lambda})$ can be interpreted as follows. Consider the twisted product space 
\begin{align*}
\widetilde{Gr}(\vec\lambda):=&\overline{Gr(\lambda^1)}\widetilde\times\overline{Gr(\lambda^2)}\widetilde\times\cdots\widetilde\times \overline{Gr(\lambda^n)}\\
=&\left\{([g_1],\ldots,[g_n])\in Gr^n\mid d(t^0,[g_1])\leq \lambda^1, d([g_i],[g_{i+1}])\leq\lambda^{i+1}\right\}.
\end{align*}
We will also need the open stratum $Gr(\vec\lambda)$ where the inequalities are replaced by equalities. The convolution morphism $l:\widetilde{Gr}(\vec\lambda)\rightarrow Gr$ sends a point to the last factor in the twisted product. The fiber over $t^0$ is called the Satake fiber, which we denote by
\begin{align*}
P(\vec{\lambda})=\left\{([g_0]=t^0=[g_n],[g_1],\ldots,[g_{n-1}]) \in Gr^n\mid d([g_{i-1}],[g_{i}])\leq\lambda^i\right\}.
\end{align*}
Then the invariant space $\Inv(\vec{\lambda})$ is identified with the top homology $H_{\text{top}}(P(\vec\lambda))$ under the geometric Satake equivalence. Note that $\widetilde{Gr}(\vec\lambda)$ has complex dimension $\langle \lvert \vec\lambda \rvert,2\rho\rangle$ and $P(\vec\lambda)$ has complex dimension $\langle \lvert \vec\lambda \rvert,\rho\rangle$. Proofs of the second part of the following theorem can be found in \cite[Prop.~3.1]{Hai1} and \cite[Lem.~4.6]{HS}.

\begin{Thm}{\cite{Lus,Gin,BD,MV}}
The category of $G^L(\mathcal O)$-equivariant perverse sheaves on $Gr$ is equivalent to the tensor category of representations of $G$. Under this equivalence the invariant space is isomorphic the top Borel--Moore homology of the Satake fiber,
\begin{align*}
\Inv(\vec\lambda)\cong H_{top}(P(\vec\lambda)).
\end{align*}
Hence, the top-dimensional components of $P(\vec\lambda)$ give a basis called the Satake basis.
\end{Thm}

Let $R:\Inv(\vec\lambda)\rightarrow \Inv(R(\vec\lambda))$ denote the linear map on invariant spaces given by rotation of tensor factors. In the next subsection we show that this agrees (up to a sign) with geometric rotation of components, which permutes the components.

\subsubsection{Indexing Components and Rotation}
When the weights $\vec\lambda$ are all minuscule, Haines \cite{Hai1} showed that the polygon space is equidimensional, so each irreducible component $Z$ gives a basis vector $[Z]$ in $H_{\text{top}}(P(\vec\lambda))$. In this case, Fontaine, Kamnitzer, Kuperberg \cite{FKK} showed that the components are indexed by minuscule paths of type $\vec\lambda$, which are equivalent to rectangular $({\bf v},\ldots,{\bf v})$-semistandard tableaux of content specified by $\vec\lambda$. Let $Z_T$ denote the component labelled by such a tableau $T$. Fontaine and Kamnitzer \cite{FK} defined the geometric rotation operator 
\begin{align*}
R_{\text{geom}}:H_{\text{top}}(P(\vec\lambda))\rightarrow H_{\text{top}}(P(R(\vec\lambda)))
\end{align*}
that maps $[Z]$ to the basis vector $R_{geom}([Z])=[Z']$ where
\begin{align*}
Z'=\left\{\left(t^0,[g_1^{-1}g_2],\ldots,[g_1^{-1}g_{n-1}],[g_1^{-1}]\right)\mid (t^0,[g_1],\ldots,[g_{n-1}])\in Z\right\}.
\end{align*}
They proved the following theorem.
\begin{Thm}[{\cite[Theorem 3.4]{FK}}]\label{thm:FK}
Let $\vec\lambda$ be a sequence of minuscule weights. 
\begin{enumerate}
\item
Geometric rotation maps a component labelled by tableau $T$ to the component labelled by the promotion tableau $\partial(T)$, i.e. $R_{geom}([Z_T])=[Z_{\partial(T)}]$. 
\item
The following diagram commutes up to a factor of $(-1)^{\langle \lambda^1,2\rho\rangle}$.
$$
\begin{tikzcd}
\Inv(\vec{\lambda}) \arrow[r, "Satake"] \arrow[d, "R"] & H_{\text{top}}(P(\vec\lambda)) \arrow[d, "R_{geom}"] \\
\Inv(R(\vec{\lambda})) \arrow[r, "Satake"] & H_{\text{top}}(P(R(\vec\lambda)))
\end{tikzcd}
$$
\end{enumerate}
\end{Thm}

Our goal is to prove the analogous theorem when $\vec\lambda$ is a sym-alt sequence. To do so we will embed $\Inv(\vec\lambda)$ into a larger space of invariants $\Inv(\vec\mu)$ where $\mu$ is a minuscule sequence, and apply the previous theorem.

Let us first consider the indexing of components. As in the minuscule case, the top-dimensional components of $P(\vec\lambda)$ are labelled by $\delta$-semistandard tableaux with $(\delta,\gamma)$ corresponding to $\vec\lambda$. To see this, recall a conjecture of Kamnitzer \cite{Kam} that was proved by Goncharov--Shen \cite{GS}, building on the work of Fock--Goncharov \cite{FG}. To state it we recall the Fock--Goncharov functions, although we will not use this definition explicitly (see \cite{Le2} for the following definition). 
\begin{Def}
Let $\vec i=(i_1,\ldots,i_n)$ be a sequence of nonnegative integers such that $\sum i_j=m$. Let $p=([g_0]=t^0,[g_1],\ldots,[g_{n-1}])$ be a point in $P(\vec\lambda)$ and let $L_j$ be a lattice representative for $[g_j]$. Let $v_{j1},\ldots,v_{ji_j}$ be elements of the lattice $L_j$ and consider the values $-val(det(v_{11},\ldots,v_{1i_1},\ldots,v_{n1},\ldots,v_{ni_n}))$ where $val$ returns the smallest power of $t$ in a Laurent series. Then define $f_{\vec i}(p)$ to be the maximum of these values as the $v_{j1},\ldots,v_{ji_j}$ range over elements of the lattice $L_j$. This is not quite well-defined because points of $Gr$ are lattices up to homothety. To fix this, add the term $\frac{1}{m}\sum_{j=1}^n i_j \cdot val(\det(L_j))$.
\end{Def}

\begin{Thm}[\cite{GS}]\label{thm:GS}
Let $\vec\lambda=(\lambda^1,\ldots,\lambda^n)$ be an arbitray sequence of dominant weights. The top-dimensional components of $P(\vec\lambda)$ are in bijection with $SL_m$ $n$-hives of type $\vec\lambda$. More precisely, each top-dimensional component $Z$ contains an open dense subset $Z^o$ on which the Fock--Goncharov functions $f_{\vec i}(p)$ are constant as $p$ varies in $Z^o$, and taken together these values form a hive. Furthermore, the successive differences of the edge values of the hive along the $(i,j)$ edge give the weight-valued distance $d([g_{i}],[g_{j}])$ for a point $p=(t^0,[g_1],\ldots,[g_{n-1}])$.
\end{Thm}
Hence, the distance $d([g_{i}],[g_{j}])$ is constant on points $p=(t^0,[g_1],\ldots,[g_{n-1}])$ in $Z^o$. We call these the \emph{generic distances of a component}. For arbitrary $\vec\lambda$ two different components can have the same generic distances for all $i,j$, but in the sym-alt setting this is impossible. Put differently, when $\vec\lambda$ is a sym-alt sequence, the generic distances, or equivalently the hive values along the edges, are enough to reconstruct the entire hive. In fact, the distances $d(t^0,[g_i])$ suffice. The following corollary is discussed in detail for the minuscule case in \cite{Akh}.
\begin{Cor}
Let $\vec\lambda$ be a sym-alt sequence and $(\delta,\gamma)$ the corresponding orientation-content pair. The generic distances $d(t^0,[g_i])$ of each top-dimensional component of $P(\vec\lambda)$ form a $\delta$-semistandard tableau with content $\gamma$. This is a bijection between components and $\delta$-semistandard tableaux with content $\gamma$.
\end{Cor}
\begin{proof}
By the previous Theorem each component corresponds to a $n$-hive and for all $i$ the distance $d(t^0,[g_i])$ is given by the differences along the $(1,i)$-edge. 

For any triangulation $T$ of the $n$-gon, the face hive values of the faces of $T$ determine the entire $n$-hive by the octahedron recurrence. Let $T$ be the fan triangulation with edges $(1,i)$ for all $i$, so that each triangle contains an external edge of the $n$-gon. Since each external edge is labelled by a $\omega_k$ or $\sigma_k$, by the Pieri rule the weights $d(t^0,[g_1]),\ldots,d(t^0,[g_n])$ form a $\delta$-semistandard tableau of content $\gamma$. Since $3$-hives count tensor product multiplicities and tensoring with $V_{\omega_k}$ or $V_{\sigma_k}$ is multiplicity free, the edge hive values along the $(1,i)$ edges uniquely determine the $3$-hives in the fan triangulation, which, as already mentioned, determine the entire $n$-hive. 
\end{proof}

Define geometric rotation as in the minuscule case. Theorem \ref{thm:GS} together with Proposition \ref{prop:main} yield the analogue of Theorem \ref{thm:FK}, part $(1)$ (and also gives an alternative proof).
\begin{Thm}\label{thm:geomrotsymalt}
Let $\vec\lambda$ be a sym-alt sequence and $(\delta,\gamma)$ the corresponding orientation-content pair. Geometric rotation $R_{geom}:H_{top}(P(\vec\lambda))\rightarrow H_{top}(P(R(\vec\lambda)))$ maps a component labelled by tableau $T$ to the component labelled by tableau $\partial_\delta(T)$, i.e. $R_{geom}([Z_T])=[Z_{\partial_\delta(T)}]$.
\end{Thm}
\begin{proof}
By Theorem \ref{thm:GS} the component $Z_T$ contains an open dense set on which the Fock--Goncharov functions form a hive. By repeated applications (of the $SL_m$ analogue) of Proposition \ref{prop:main} the generic distances $d([g_1],[g_i])$ form the $\delta$-semistandard tableau $\partial_\delta(T)$ with content given by $R(\vec\lambda)$. Since the distance function is $G(\mathcal K)$-invariant, the component $R_{geom}(Z_T)$ is $Z_{\partial_\delta(T)}$.
\end{proof}

Now that we have established that components of $P(\vec\lambda)$ are labelled by $\delta$-semistandard tableaux and geometric rotation corresponds to $\delta$-promotion, our goal is to prove the analogue of Theorem \ref{thm:FK}, part (2). \begin{Thm}\label{thm:commutes}
Let $\vec\lambda=(\lambda^1,\ldots,\lambda^n)$ be a sym-alt sequence. The following diagram commutes up to a factor of $(-1)^{\langle \lambda^1,2\rho\rangle}$.
$$
\begin{tikzcd}
\Inv(\vec{\lambda}) \arrow[r, "Satake"] \arrow[d, "R"] & H_{\text{top}}(P(\vec\lambda)) \arrow[d, "R_{geom}"] \\
\Inv(R(\vec{\lambda})) \arrow[r, "Satake"] & H_{\text{top}}(P(R(\vec\lambda)))
\end{tikzcd}
$$
\end{Thm}
\begin{proof}
Let $\vec\mu$ be the same sequence as $\vec\lambda$, but with each $\sigma_k$ replaced with $k$ many $\omega_1$'s. More precisely, $\vec\mu=((\mu^{11},\ldots,\mu^{1k_1}),\ldots,(\mu^{n1},\ldots,\mu^{nk_n}))$ where $k_j=1$ and $\mu^{j1}=\lambda^j$ if $\lambda^j$ is minuscule and $\mu^{j1}=\cdots=\mu^{jk}=\omega_1$ if $\lambda^j=\sigma_k$. In particular, $\lvert \vec\lambda\rvert=\lvert\vec\mu\rvert$. We have an inclusion of invariants $\psi:\Inv(\vec\lambda)\rightarrow \Inv(\vec\mu)$. We also have the following projection morphism 
\begin{align*}
\phi:\widetilde{Gr}(\vec\mu)=\left(Gr_{\mu^{11}}\widetilde\times\cdots\widetilde\times Gr_{\mu^{1{k_1}}}\right)\widetilde\times\cdots\widetilde\times \left(Gr_{\mu^{n1}}\widetilde\times\cdots\widetilde\times Gr_{\mu^{n{k_n}}}\right)\rightarrow \widetilde{Gr}(\vec\lambda), 
\end{align*}
which commutes with the convolution morphisms $l$ and $m$.
\begin{equation*}
\begin{tikzcd}
\widetilde{Gr}(\vec\mu)  \arrow[d, "\phi"] \arrow[dr, "m"]&\\
\widetilde{Gr}(\vec\lambda) \arrow[r,"l"]& Gr_{\lvert \vec\lambda\rvert}
\end{tikzcd}
\end{equation*}
The morphism $\phi$ restricts to $P(\vec\mu)\rightarrow P(\vec\lambda)$ and induces a pushforward on homology $\phi_*:H_{top}(P(\vec\mu))\rightarrow H_{top}(P(\vec\lambda))$. Note that $\widetilde{Gr}(\vec\mu)$ is smooth, and $\phi$ is birational and an isomorphism over the open stratum $Gr(\vec\lambda)$. Set $U=\phi^{-1}(Gr(\vec\lambda))$. The map $\phi^{-1}$ restricted to $Gr(\vec\lambda)$ is easy to describe. Suppose $\lambda^j=\sigma_k$, so that for a point $g=(t^0,[g_1],\ldots,[g_{n-1}])$ we have $d([g_{j-1}],[g_j])=\sigma_k$. Then there is a unique set of points $[h_{j1}],[h_{j2}]\ldots,[h_{jk+1}]$ such that $d([h_{ji}],[h_{ji+1}])=\omega_1$ for all $i$ and $[h_{j1}]=[g_{j-1}]$ and $[h_{jk+1}]=[g_j]$.

Haines \cite{Hai2} showed that every top-dimensional component of $P(\vec\lambda)$ intersects the open stratum $Gr(\vec\lambda)$, so any component of $P(\vec\mu)$ that intersect $U$ has an open dense set that maps isomorphically to an open dense subset of a top-dimensional component of $P(\vec\lambda)$. All other components of $P(\vec\mu)$ map to lower dimensional components of $P(\vec\lambda)$, so the map $\phi_*:H_{top}(P(\vec\mu))\rightarrow H_{top}(P(\vec\lambda))$ sends any component disjoint from $U$ to zero. Letting $W$ be the subspace of $H_{top}(P(\vec\mu))$ spanned by the Satake basis vectors corresponding to components that meet $U$, we see that $\phi_*$ is an isomorphism between $W$ and $H_{top}(P(\vec\lambda))$. We'll think of $\phi_*$ as restricted to $W$. 

Suppose that $\lambda^1=\sigma_k$ (or if $\lambda^1$ is minuscule, then let $k=1$). We claim that the left face of the following cube commutes because the other faces commute (up to appropriate sign). The maps coming out of the page are from the geometric Satake isomorphism.

\begin{equation*}
\begin{tikzcd}[row sep=scriptsize, column sep=scriptsize]
& H_{\text{top}}(P(\vec\lambda)) \arrow[dl] \arrow[rr, "\phi_*^{-1}" near start] \arrow[dd, "R_{geom}" near start] & & W\subset H_{\text{top}}(P(\vec\mu)) \arrow[dl] \arrow[dd, "R_{geom}^k" near start] \\
\Inv(\vec\lambda) \arrow[rr, "\psi" near end] \arrow[dd, "R" near start] & & \Inv(\vec\mu) \\
& H_{\text{top}}(P(R(\vec\lambda))) \arrow[dl] \arrow[rr, "\phi_*^{-1}" near start] & & W'\subset H_{\text{top}}(P(R^k(\vec\mu))) \arrow[dl] \\
\Inv(R(\vec\lambda)) \arrow[rr, "\psi" near end] & & \Inv(R^k(\vec\mu)) \arrow[from=uu,  "R^k" near start]\\
\end{tikzcd}
\end{equation*}

Note that the front face commutes since rotating tensor factors $k$ times preserves direct summands within those $k$ factors. The right face is an instance of the Theorem \ref{thm:FK} applied $k$ times, so it commutes up to $(-1)^{\langle \mu^1+\cdots+\mu^k, 2\rho\rangle}=(-1)^{\langle \lambda^1, 2\rho\rangle}$. The back face commutes by the description of $\phi^{-1}$ restricted to the open stratum $Gr(\vec\lambda)$. That the top and bottom faces commute is a bit more subtle and is proved in \cite[Lem.~4.2]{Fon}.
\end{proof}

Note that the above argument can be applied to an arbitrary sequence of dominant weights $\vec{\nu}$ where in this case $\nu^1$ is decomposed as a sum of fundamental weights, not necessarily all $\omega_1$. As we were finishing this manuscript, Joel Kamnitzer pointed out the recent preprint of Baumann, Gaussent, and Littelmann to us. In \cite{BGL20}, Theorem 6.1, they prove that geometric rotation takes the Satake basis of $\Inv (\vec \nu)$ to $\Inv (R(\vec \nu))$ for an arbitrary sequence of $n$ weights $\vec \nu$. Their proof uses the crystal structure on $V_{\nu^1}\otimes \ldots \otimes V_{\nu^n}$, and, as far as we could tell, uses ideas different from ours, so we decided to still include our proof of Theorem \ref{thm:geomrotsymalt} above.

\subsection{Diagonalization Via the Fusion Product}
The diagonalization of the action of $R^r$ follows exactly as in \cite[\S~4.2]{FK} (see also \cite[\S~8.3]{Wes}), so we will be brief. Consider the current algebra $\mathfrak{sl_m}[t]=\mathfrak{sl_m}\otimes \mathbb C[t]$. To each $\mathfrak sl_m$ representation $V$ and $z\in\mathbb C$ is associated the evaluation representation $V^z$ of $\mathfrak{sl_m}[t]$ where $X\otimes f(t)\cdot v=f(z)X\cdot v$ for $X\in\mathfrak{sl}_m$, $f(t)\in \mathbb C[t]$, $v\in V$. 
We need the graded tensor product, usually called the fusion product, as defined by Feigin and Loktev \cite{FL}. Fix a sym-alt sequence $\vec\lambda$ and for each $i$ let $v_i\in V_{\lambda^i}$ be a highest weight vector. Consider the $\mathfrak sl_m[t]$ representation $V_{\lambda^1}^{z_1}\otimes\cdots\otimes V_{\lambda^n}^{z_n}$ for $z_i\in \mathbb C$ where $\mathfrak{sl}_m[t]$ acts via the coproduct. For pairwise distinct $z_i$ the vector $v_1\otimes \cdots\otimes v_n$ is a cyclic vector, so that $U(\mathfrak{sl}_m[t])\left(v_1\otimes\cdots\otimes v_n\right)\cong V_{\lambda^1}^{z_1}\otimes\cdots\otimes V_{\lambda^n}^{z_n}$.

The algebra $U(\mathfrak{sl_m}[t])$ is graded by the degree of $t$ with filtered piece $U^{\leq k}$ consisting of degree $k$ or less. This induces a filtration on $V_{\lambda^1}^{z_1}\otimes\cdots\otimes V_{\lambda^n}^{z_n}$ given by $\mathcal F^{\leq k}=U^{\leq k}\left( v_1\otimes\cdots\otimes v_n\right)$ where we set $\mathcal F^{-1}=0$. Then define the fusion product to be the associated graded of this filtration:
\begin{align*}
V_{\lambda^1}^{z_1}\star\cdots \star V_{\lambda^n}^{z_n}:=\oplus \mathcal F^{\leq k}/\mathcal F^{\leq k-1}.
\end{align*}
Each graded component is an $\mathfrak sl_m$ representation and as an $\mathfrak sl_m$ representation we have the decomposition
\begin{align*}
V_{\lambda^1}^{z_1}\star\cdots \star V_{\lambda^n}^{z_n}\cong \oplus_\mu \oplus_{n\geq 0} V_\mu^{\oplus K_{\mu,\vec\lambda}[n]},
\end{align*}
a graded version of the tensor product $V_{\lambda^1}\otimes\cdots\otimes V_{\lambda^n}$. Note that Feigin and Loktev conjectured that the multiplicities do not depend on the parameters as long as they are pairwise distinct, which was proved for multiples of fundamental weights in \cite{AKS}. The isotypic component for each $\mu$ is graded and we define the polynomial $K_{\mu,\vec\lambda}(q)=\sum_n K_{\mu,\vec\lambda}[n]q^n$. It was shown in \cite{AKS} that this polynomial agrees with the generalized Kostka polynomial of \cite{SW,SS,KS}.

Let $\vec\lambda$ be a sym-alt sequence and $r$ is the smallest integer such that $R^r(\vec\lambda)=\vec\lambda$ and set $l=n/r$. Let $W$ be the corresponding fusion product with cyclic vector $v_1\otimes\cdots\otimes v_n$. There is a rotation map $R^r:W\rightarrow W$ given by rotation of tensor factors $r$ times, $R^r(v_1\otimes\cdots\otimes v_n)=v_{r+1}\otimes\cdots v_{r-1}\otimes v_r$. Let $\zeta$ be a primitive $l$th root of unity.

\begin{Prop}[\cite{FK,Wes}]
Let $r$ be as above, $l=n/r$, and $\eta$ a primitive $n$th root of unity such that $\eta^r=\zeta$. Let $z_i=\eta^i$. With the notation as above, $R^r$ acts on the $k$th graded piece as $\eta^{kr}=\zeta^k$.
\end{Prop}
\begin{proof}
This is a direct computation.
\begin{align*}
R^r\left(Xt^k v_1\otimes\cdots\otimes v_n\right)&=R^r\left(\sum \eta^{kj}v_1\otimes\cdots\otimes Xv_j\otimes\cdots\otimes v_n\right)\\
&=\sum \eta^{kj}v_{r+1}\otimes\cdots \otimes Xv_j\otimes\cdots\otimes v_r\\
&=\eta^{kr} (Xt^k R^r(v_1\otimes\cdots\otimes v_n))\\
&=\eta^{kr} (Xt^k v_1\otimes\cdots\otimes v_n)\\
\end{align*}
\end{proof}

Putting everything together, we see that promotion on $RT_m(\delta,\gamma)$ corresponds to the linear map $(-1)^{\langle\lambda^1+\ldots+\lambda^r,2\rho\rangle}R^r$ on $\Inv(\vec\lambda)$. The trace of $R^{dr}$ on $\Inv(\vec\lambda)$ is $K_{\nu,\vec\lambda}(\zeta^d)$. Hence the trace of $(-1)^{d\langle\lambda^1+\ldots+\lambda^r,2\rho\rangle}R^{dr}$ is 
\begin{align*}
(-1)^{d\langle\lambda^1+\ldots+\lambda^r,2\rho\rangle}K_{\mu,\vec\lambda}(\zeta^d)&=(\zeta^{l/2})^{d\langle\lambda^1+\ldots+\lambda^r,2\rho\rangle}K_{\mu,\vec\lambda}(\zeta^d)\\
&=(\zeta^d)^{\langle\lvert\vec\lambda\rvert,\rho\rangle}K_{\mu,\vec\lambda}(\zeta^d).
\end{align*}
Hence, the cyclic sieving polynomial is $q^{\langle\lvert\vec\lambda\rvert,\rho\rangle}K_{\mu,\vec\lambda}(q)$.

\begin{Rem}
We conclude with the observation that the proofs of sections 4.1 and 4.2 go through without change when $\vec\lambda$ is a sequence of arbitrary dominant weights (in the proof of Theorem \ref{thm:commutes} decompose each $\lambda^i$ as a sum of minuscule weights given by its columns). In this case, there is a cyclic sieving phenomenon where the set $X$ is the set of hives with boundary $\vec\lambda$ and the polynomial is an appropriate $q$-multiple of the corresponding graded multiplicity polynomial defined by the fusion product. It would be interesting if there is a statistic on hives computing these polynomials. 
\end{Rem}

\subsection{Example Computation with the Fusion Product}

As an example, we'll use the fusion product to compute a generalized Kostka polynomial. Let $\mathfrak{g}=\mathfrak{sl}_2$ with the usual basis $e,f,h$, let $\vec\lambda=((1,0),(1,0),(1,0),(1,0))$ and $\mu=(2,2)$. Note that this example is already covered by Rhoades's result \cite{Rho}. Then $K_{\mu,\vec \lambda}(q)=q^2+q^4$. We will show that this is equal to the graded multiplicity of the trivial representation in
\[
  V(1)\star V(1)\star V(1) \star V(1)
\]
Let $\{v_+,v_-\}$ be the standard basis of $V(1)\cong \mathbb{C}^2$, so that $f(v_+)=v_-$. To compute the fusion product, we will use the filtration coming from the highest weight vector \mbox{$v_+\otimes v_+ \otimes v_+ \otimes v_+$} and the vector $z=(0,1,2,3)$. We will use the shorthand $v_{pqrs}$ to mean the vector $v_p\otimes v_q\otimes v_r\otimes v_s$, so the highest weight vector is $v_{++++}$.

Consider the $0$-weight space of the tensor product $V(1)^{\otimes 4}$. It is $6$-dimensional, spanned by the vectors
\[
  \mathcal{B}=\left\{ v_{--++}, v_{-+-+}, v_{-++-}, v_{+--+}, v_{+-+-}, v_{++--}  \right\}.
\]
To find the space of invariants, we are looking for vectors killed by $e\in \mathfrak{sl}_2$. The two-dimensional invariant space is spanned by the vectors (expressed in the basis $\mathcal{B}$)
\[
  \begin{bmatrix}
    1 & 0 & -1 & -1 & 0 & 1
  \end{bmatrix}
  \text{ and }
  \begin{bmatrix}
    0 & 1 & -1 & -1 & 1 & 0
  \end{bmatrix}.
\]
To reach the $0$-weight space from the high weight vector $v_{++++}$, we will need to apply elements of the form $(p(t)f)$ twice. We compute the images of these maps explicitly, and express them in the $\mathcal{B}$-basis.
\begin{align*}
  (f)(f)(v_{++++}) &=
                     \begin{bmatrix}
                       2 & 2 & 2 & 2 & 2 & 2
                     \end{bmatrix}\\
  (tf)(f)(v_{++++})=(f)(tf)(v_{++++}) &=
                                        \begin{bmatrix}
                                          1 & 2 & 3 & 3 & 4 & 5
                                        \end{bmatrix}\\
  (t^2f)(f)(v_{++++})=(f)(t^2f)(v_{++++}) &=
                                            \begin{bmatrix}
                                              1 & 4 & 9 & 5 & 10 & 13
                                            \end{bmatrix}\\
  (tf)(tf)(v_{++++})&=
                      \begin{bmatrix}
                        0 & 0 & 0 & 4 & 6 & 12
                      \end{bmatrix}\\
  (t^3f)(f)(v_{++++})=(f)(t^3f)(v_{++++})&=
                                           \begin{bmatrix}
                                             1 & 8 & 27 & 9 & 28 & 35
                                           \end{bmatrix}\\
  (t^2f)(tf)(v_{++++})=(tf)(t^2f)(v_{++++})&=
                                             \begin{bmatrix}
                                               0 & 0 & 0 & 6 & 12 & 30
                                             \end{bmatrix}\\
  (t^3f)(tf)(v_{++++})=(tf)(t^3f)(v_{++++})&=
                                             \begin{bmatrix}
                                               0 & 0 & 0 & 10 & 30 & 78
                                             \end{bmatrix}\\
  (t^2f)(t^2f)(v_{++++})&=
                          \begin{bmatrix}
                            0 & 0 & 0 & 8 & 18 & 72
                          \end{bmatrix}                                                                    
\end{align*}
We see that the span of the vectors  $\left\{(f)(f)(v_{++++}),(tf)(f)(v_{++++}) \right\}$ does not intersect the invariant space, but in degree $2$ we have
\begin{align*}
  &11(f)(f)(v_{++++})-18(f)(tf)(v_{++++})+3(f)(t^2f)(v_{++++})+3(tf)(tf)(v_{++++})  \\
  =&\begin{bmatrix}
    7 & -2 & -5 & -5 & -2 & 7
  \end{bmatrix}
\end{align*}
which we see is in the invariant space. By direct computation, the intersection of $\mathcal{U}^{\leq 3}$ with the invariant space is still $1$-dimensional, but notice that the set
\begin{align*}
  &\left\{ (f)(f)(v_{++++}),(tf)(f)(v_{++++}), (t^2f)(f)(v_{++++}),(tf)(tf)(v_{++++})  \right. \\ & \left.  (t^3f)(f)(v_{++++}), (t^3f)(tf)(v_{++++}) \right\}
\end{align*}
is linearly independent, therefore the other copy of the trivial representation is in degree $4$.

\bibliographystyle{alpha}
\bibliography{biblio}
 
\end{document}